\documentclass[11pt]{article}          
\usepackage{amsfonts,amsmath,latexsym,amssymb} 
\usepackage{theorem}                
\usepackage{mathrsfs,upref}         
\usepackage{mathptmx}		    
\usepackage{epsfig,here,color}
\usepackage{float,epsfig,mathdots}
\usepackage{floatflt}
\usepackage{fancyhdr}
\usepackage{supertabular}
\usepackage{hyperref}
 \usepackage{multirow}
\usepackage{graphics}
\usepackage{graphicx}
\usepackage{epstopdf}               	            %
\usepackage{tikz}

\newenvironment{proof}{\paragraph{Proof:}}{\hfill$\square$ \vskip 0.12cm}
\newtheorem{theorem}{Theorem}[section]
\newtheorem{lemma}[theorem]{Lemma}
\newtheorem{corollary}[theorem]{Corollary}

\theoremstyle{definition}

\newtheorem{definition}[theorem]{Definition}
\newtheorem{example}[theorem]{Example}

\newtheorem{remark}[theorem]{Remark}

\numberwithin{equation}{section}

\def \F{\mathbb{F}}
\def \L{\mathbb{L}}

\def \rank{\mathrm{rank}}

\newcommand{\norm}[1]{\left\lVert#1\right\rVert}
\newcommand{\normp}[1]{|\!|\!|{#1}|\!|\!|}
\newcommand{\columnshiftsum}{\begin{tikzpicture}[scale=0.24]
\draw (0,0) rectangle (1,1);
\draw[->] (0,.5) -- (1.6,.5);
\draw (.5,0) -- (.5,1);
\end{tikzpicture}}
\newcommand{\rowshiftsum}{\begin{tikzpicture}[scale=0.24,baseline]
\draw (0,0) rectangle (1,1);
\draw (0,.5) -- (1,.5);
\draw[<-](.5,-.6) -- (.5,1);
\end{tikzpicture}}

\begin{document}

\title{Vector Spaces of Generalized Linearizations for Rectangular Matrix Polynomials}


\author{Biswajit Das~\thanks{%
                 Department of Mathematics,
                Indian Institute of Technology Guwahati,
                Guwahati 781039, Assam, India,
               (\texttt{biswajit.das@iitg.ac.in, shbora@iitg.ac.in})}
\and
 Shreemayee Bora~\footnotemark[1]}
\date{}
\maketitle

\begin{abstract}
        The seminal work~\cite{MacMMM06a} introduced vector spaces of matrix pencils, with the property that almost all the pencils
in the spaces are strong linearizations of a given square regular matrix polynomial. This work was subsequently extended to include the case of square singular
matrix polynomials in~\cite{DeDM09}. We extend this work to non-square matrix polynomials by proposing similar vector spaces of rectangular matrix pencils that are equal to the ones in~\cite{MacMMM06a} when the polynomial is square. Moreover, the properties of these vector spaces are similar to those in~\cite{DeDM09} for the singular case. In particular, the complete eigenvalue problem associated with the matrix polynomial can be solved by using almost every matrix pencil from these spaces. Further, almost every pencil in these spaces can be `trimmed' to form many smaller pencils that are strong linearizations of the matrix polynomial which readily solve the complete eigenvalue problem for the polynomial. These linearizations are easier to construct and are often smaller than the Fiedler linearizations introduced in~\cite{DeDM12}. Further, the global backward error analysis in~\cite{DopLPV16} applied to these linearizations, shows that they provide a wide choice of linearizations with respect to which the complete polynomial eigenvalue problem can be solved in a globally backward stable manner.
\end{abstract}



\section{Introduction}~\label{intro}

Eigenvalue problems associated with matrix polynomials $P(\lambda) = \sum_{i=0}^k\lambda^iA_i,$ where $A_i, i = 0, \ldots, k$ are $m \times n$ real or complex matrices, occur in a wide range of applications like vibration analysis of machines, building and vehicles, in control theory and linear systems theory and as approximate solutions of other nonlinear eigenvalue problems~\cite{TisM01, GohLR82, Kai80, Ros70, Var91}.

When the polynomial is square and regular, i.e.,  $\det \, P(\lambda) \not\equiv 0,$ the associated polynomial eigenvalue problem consists of finding the finite and infinite eigenvalues and corresponding eigenvectors. However when the polynomial is singular, i.e., when it is either non-square or $\det \, P(\lambda) \equiv 0,$ then the eigenvalue problem is said to be a complete eigenvalue problem as in addition to finite and infinite eigenvalues and corresponding elementary divisors, the minimal indices and bases corresponding to the left and right null spaces of the polynomial also have to be computed. The most common approach for solving such problems is to linearize them by converting the problem into an equivalent problem associated with a larger matrix pencil of the form $L(\lambda) = \lambda X + Y$ called a linearization of $P(\lambda),$ and solving the eigenvalue problem for $L(\lambda)$ by using standard algorithms like the QZ algorithm~\cite{GolV13} when $L(\lambda)$ is regular, or the staircase algorithm~\cite{Van79} when $L(\lambda)$ is singular. The solution for $P(\lambda)$ is then recovered from that of its linearization.  The solution of the complete eigenvalue problem for singular matrix polynomials is more challenging as ideally, there should be a simple rule for extracting left and right minimal bases and indices of the polynomial from those of its linearization. We refer to~\cite{GohLR82} and a more recent survey article~\cite{MacMT15} for the theory of polynomial eigenvalue problems and their solutions.

The most commonly used forms of linearizations for solving polynomial eigenvalue problems associated with $P(\lambda) = \sum_{i=0}^k\lambda^iA_i$ are the first and second Frobenius companion forms $C_1(\lambda)$ and $C_2(\lambda),$ given by~\eqref{complin1} and \eqref{complin2} respectively. One of the first systematic studies of linearizations  to be undertaken was~\cite{MacMMM06a} which introduced the following vector spaces $\L_1(P)$ and $\L_1(P)$ of matrix pencils for a given $n \times n$ regular matrix polynomial $P(\lambda)$ as sources of linearizations of $P(\lambda).$

\begin{eqnarray}
\mathbb{L}_1(P) &:=& \{L(\lambda):L(\lambda)(\Lambda_k(\lambda)\otimes I_n)=v\otimes P(\lambda),v\in\mathbb{F}^k\}, \label{l1reg} \\
\mathbb{L}_2(P) &:=& \{L(\lambda):(\Lambda_k(\lambda)^T\otimes I_n)L(\lambda)=w^T\otimes P(\lambda),w\in\mathbb{F}^k\} \label{l2reg}
\end{eqnarray}
\begin{equation}\label{dualminb1}
\mbox{ where }\Lambda_k(\lambda) := [\lambda^{k-1},\dots,\lambda,1]^T.
\end{equation}

The defining identities in~\eqref{l1reg} and \eqref{l2reg} are called the right and left ansatz equations respectively and the corresponding vectors $v$ in~\eqref{l1reg} and the vector $w$ in~\eqref{l2reg} are called right and left ansatz vectors.
This work gave a whole new direction to research in the theory of linearizations due to the special properties of these vector spaces. For instance, it was shown that constructing pencils in these spaces corresponding to a given ansatz vector is very simple and almost all the resulting pencils are linearizations of $P(\lambda)$ from which the eigenvalues and corresponding eigenvectors can be easily recovered. In further work~\cite{HigMT09,MacMMM06b,FasMMS08}, it was shown that if $P(\lambda)$ has some special structure like, Hermitian, symmetric, $\star$-alternating and $\star$-palindromic, (see~\cite{MacMT15} for definitions), then there exist subspaces of $\L_1(P)$ and $\L_2(P)$ with the property that almost every pencil of the subspace is a structure preserving linearization of $P(\lambda)$ from which both finite and infinite eigenvalues of $P(\lambda)$ and corresponding eigenvectors can be easily recovered.

It was shown in~\cite{DeDM09} that even when $P(\lambda)$ is square but singular, almost every pencil in $\L_1(P)$ and $\L_2(P)$ is a linearization of $P(\lambda)$ from which the solution of complete eigenvalue problem for $P(\lambda)$ can be easily recovered. The vector space setting for constructing linearizations has since been extended to cover other polynomial bases~\cite{FasS17} and inspired further work that throws fresh light on these spaces~\cite{NakNT17}. Other important choices of linearizations not covered by $\L_1(P)$ and $\L_2(P)$ are the Fiedler pencils and their generalizations~\cite{AntV04,VolA11,DeDM10,BueDD11,BueF12} which are also sources of linearizations for non-square matrix polynomials~\cite{DeDM12}. Systematic studies of linearizations that cover both square and non-square linearizations are relatively recent in the literature. For example, \cite{DopLPV16} introduced the framework of block minimal bases pencils as potential linearizations of rectangular matrix polynomials with focus on particular subclasses like the Block Kronecker pencils. These ideas were further extended in~\cite{BueDPSZ16}. Inspired by~\cite{MacMMM06a}, the recent work~\cite{FasS18} considers linearizations of rectangular matrix polynomials in a vector space setting. Referred to as Block Kronecker ansatz spaces, these vector spaces contain Block Kronecker linearizations as well as Fiedler linearizations and their extensions modulo permutations and share some of the important properties that $\L_1(P)$ and $\L_2(P)$ have when $P(\lambda)$ is square. However, the Block Kronecker ansatz spaces do not become $\L_1(P)$ and $\L_2(P)$ when $P(\lambda)$ is square.

 The goal of this present work is to provide a direct generalization of the spaces $\L_1(P)$ and $\L_2(P)$ to the case when $P(\lambda)$ is not square by forming vector spaces of matrix pencils that have some of the key features of $\L_1(P)$ and $\L_2(P)$ and coincide with them when $P(\lambda)$ is square. We propose such vector spaces and show that the matrix pencils in these spaces can be constructed from the coefficient matrices of $P(\lambda)$ in a manner very similar to the ones in $\L_1(P)$ and $\L_2(P).$ We also show that the solution of the complete eigenvalue problem for $P(\lambda)$ can be easily recovered from that of almost every pencil in these spaces. To this end, we define generalized linearizations of matrix polynomials (which we refer to in short as g-linearizations), and their strong versions and show that the proposed vector spaces have all the properties with respect to being g-linearizations that $\L_1(P)$ and $\L_2(P)$ are shown to possess with respect to being linearizations of square singular polynomials in~\cite{DeDM09}.

 Although the pencils in our proposed vector spaces are not linearizations of the non-square polynomial $P(\lambda)$ in the conventional sense, we show that almost every such pencil in these spaces can give rise to many linearizations of $P(\lambda)$ from which the finite and infinite eigenvalues and corresponding elementary divisors as well as left and right minimal indices and bases of $P(\lambda)$ can be easily extracted. We also give the relationship between these linearizations and those in some of the Block Kronecker ansatz spaces in~\cite{FasS18}, thus showing how g-linearizations and linearizations arising from them, interact with some of the important linearizations for rectangular matrix polynomials in the literature.

 From the point of view of computation, a desirable property of any linearization for solving an eigenvalue problem associated with a matrix polynomial $P(\lambda)$ is that the computed solution is the exact solution of some polynomial $P(\lambda) + \Delta P(\lambda)$ such that the ratio $\frac{\normp{\Delta P}}{\normp{P}}$ is of the order of unit roundoff ${\bf u}$ with respect to some choice of norm $\normp{ \cdot}$ on matrix polynomials. Moreover, when $P(\lambda)$ is singular, it is also desirable that the rules for extracting the left and right minimal indices of $P(\lambda)$ from a particular class of linearizations for $P(\lambda)$ remains the same for $P(\lambda) + \Delta P(\lambda)$ with respect to that class. This is referred to as global backward stability analysis for the polynomial eigenvalue problem and has been undertaken for algorithms that use the Frobenius companion linearizations in~\cite{VanD83}. More recently this has been extended to the Block Kronecker linearizations in~\cite{DopLPV16} which identifies optimal choices of Block Kronecker linearizations that ensure global backward stability when used to solve the eigenvalue problem for $P(\lambda).$ We extend the analysis in~\cite{DopLPV16} to the linearizations of $P(\lambda)$ extracted from g-linearizations. Our analysis shows that there is a wider choice of linearizations beyond the ones identified in~\cite{DopLPV16} that can be used to solve the complete eigenvalue problem for $P(\lambda)$ in a backward stable manner.

\section{Definitions and notations}\label{basic}

In this paper we use standard notations like  $\mathbb{F}$ to denote the field of real or complex numbers,  $I_n$ to denote the $n \times n$ identity matrix and $e_i, \, 1 \leq i \leq n,$ to denote the $i$-th column of $I_n$ unless otherwise specified. Also, wherever it is necessary to emphasize the dimension of a zero matrix, we will use $0_n$ to denote a column of $n$ zeros and $0_{m \times n}$ to denote the $m \times n$ zero matrix. We will use $\mathbb{F}(\lambda)$ to denote the field of rational functions with coefficients in $\mathbb{F}$ and $\mathbb{F}(\lambda)^{n}$ to denote the vector space of $n$-tuples with entries from $\mathbb{F}(\lambda)$. Also $\mathbb{F}[\lambda]$ will denote the ring of polynomials over the field $\mathbb{F}$ and $\mathbb{F}[\lambda]^{m \times n}$ will denote the ring of $m \times n$ matrix polynomials with entries from $\mathbb{F}[\lambda].$

Here we will consider $m\times n$ matrix polynomials with grade $k$ of the form
$$P(\lambda)=\displaystyle\sum_{i=0}^k \lambda^i A_i \in \mathbb{F}[\lambda]^{m\times n},$$
where any of the coefficient matrices may be the zero matrix. Degree of $P(\lambda)$ denoted by $\mathrm{deg} \, P$ is the maximum integer $d$ such that $A_d\neq 0$. In this paper we will assume that $\mathrm{deg} \, P \geq 2.$ A square matrix polynomial $Q(\lambda)$ is said to be unimodular if its determinant is a nonzero constant independent of $\lambda.$

The normal rank of $P(\lambda)$, denoted by $\mathrm{nrank}P(\lambda)$, is the rank of $P(\lambda)$ considered as a matrix with entries in $\mathbb{F}(\lambda)$. Also the $k$-reversal $\mathrm{rev}_kP(\lambda)$ of  $P(\lambda)$ is defined by $\mathrm{rev}_kP(\lambda)=\lambda^kP(1/\lambda)$.

A finite eigenvalue of $P(\lambda)$ is an element $\lambda_0\in\mathbb{F}$ such that $\mathrm{rank}P(\lambda_0)<\mathrm{nrank}P(\lambda)$. We say that $P(\lambda)$ with grade $k$ has an infinite eigenvalue if the $k-$reversal polynomial $\mathrm{rev}_kP(\lambda)=\lambda^kP(1/\lambda)$ has zero as an eigenvalue.

The following subspaces associated with $P(\lambda)$ will be frequently used.

\begin{definition}
 The right and left null spaces of a $m\times n$ matrix polynomial $P(\lambda)$, denoted by $N_r(P)$ and $N_l(P)$ respectively are defined as follows.
  \begin{eqnarray*}
 N_r(P) & = & \{x(\lambda)\in\mathbb{F}(\lambda)^{n}:P(\lambda)x(\lambda)\equiv 0\},\\
 N_l(P) & = & \{y(\lambda)\in\mathbb{F}(\lambda)^{m}:y(\lambda)^TP(\lambda) \equiv 0\}.
\end{eqnarray*}
\end{definition}

A vector polynomial is a vector whose entries are polynomials. For any subspace of $\mathbb{F}(\lambda)^{n}$, it is always possible to find a basis consisting entirely of vector polynomials. The degree of a vector polynomial is the greatest degree of its components, and the order of a polynomial basis is defined as the sum of the degrees of its vectors. Also any subspace of $\mathbb{F}(\lambda)^n$ has a polynomial basis of least order among all such bases and the ordered list of degrees of the vector polynomials in any such basis is always the same~\cite{For75}. A minimal basis of the subspace is therefore defined as any polynomial basis of least order among all such bases and the minimal indices of the subspace are the ordered list of degrees of the vector polynomials in such a basis. In particular we have the following definitions.

\begin{definition} For a given $m\times n$ matrix polynomial $P(\lambda),$ a left minimal basis is a minimal basis of $N_l(P)$ and a right minimal basis is a minimal basis of $N_r(P).$
\end{definition}

\begin{definition} For a given $m\times n$ matrix polynomial $P(\lambda),$ let $\{ x_1(\lambda), \ldots, x_p(\lambda)\}$ be a right minimal basis and $\{y_1(\lambda), \ldots, y_q(\lambda)\}$ be a left minimal basis such that $$\mathrm{deg} \, x_1 \leq \cdots \leq \mathrm{deg} \, x_p \mbox{ and } \mathrm{deg} \, y_1 \leq \cdots \leq \mathrm{deg} \, y_q.$$ Setting $\eta_i = \mathrm{deg} \, x_i, i = 1, \ldots, p,$ and $\epsilon_j = \mathrm{deg} \, y_j,$ $ j = 1, \ldots, q,$ the right and left minimal indices of $P(\lambda)$ are defined as $\eta_1 \leq  \cdots \leq \eta_p,$ and $\epsilon_1 \leq \cdots \leq \epsilon_q$ respectively.
\end{definition}

The left and right minimal bases and indices of a matrix polynomial $P(\lambda)$ are defined as the minimal bases and indices of its left  and right null spaces $N_l(P)$ and $N_r(P)$ respectively.

\noindent The most widely used approach for solving polynomial eigenvalue problems is linearization.

\begin{definition}[Linearization]
 A matrix pencil $L(\lambda)=\lambda X+Y$ with $X,Y\in \mathbb{C}^{(m+s) \times (n+s)} $ is a linearization of an $m \times n$ matrix polynomial $P(\lambda)$ of grade $k$ if there exist two unimodular matrix polynomials $E(\lambda)\in\mathbb{F}[\lambda]^{(m+s) \times (m+s)}$ and $F(\lambda)\in \mathbb{F}[\lambda]^{(n+s) \times (n+s)}$ for some positive integer $s$ such that
 $$E(\lambda)L(\lambda)F(\lambda)= \left[\begin{array}{cc} P(\lambda) & \\  & I_s\end{array}\right].$$
\end{definition}

For example the first and second Frobenius companion forms $C_1(\lambda)$ and $C_2(\lambda)$ given by

\begin{eqnarray}
C_1(\lambda) & := & \lambda \begin{bmatrix}A_k & & &\\ &I_n& & \\ & & \ddots &\\ & & & I_n\end{bmatrix} + \begin{bmatrix}A_{k-1} &A_{k-2} &\dots &A_0\\ -I_n& & &0\\ & \ddots&  &\\ & & -I_n& 0\end{bmatrix} \label{complin1}\\
C_2(\lambda) & := & \lambda \begin{bmatrix}A_k & & &\\ &I_m& & \\ & & \ddots &\\ & & & I_m\end{bmatrix} + \begin{bmatrix}A_{k-1} &-I_m &\dots & \\ A_{k-2}& & &\\ & \ddots&  &-I_m \\ A_0&0 & & 0\end{bmatrix} \label{complin2}
\end{eqnarray}
are linearizations of $P(\lambda)$ with $s = (k-1)n$ and $s = (k-1)m$ respectively. It is clear that a matrix polynomial and its linearization has the same finite eigenvalues and corresponding elementary divisors (for details, see, \cite{GohLR82}). However, if the same is to be guranteed for the eigenvalue at infinity also, then the linearization has to be a strong linearization of $P(\lambda).$

\begin{definition}[Strong Linearization]
A linearization $L(\lambda) = \lambda X+Y$ of a matrix polynomial $P(\lambda)$ of grade $k$ is called a strong linearization of $P(\lambda)$  if $\mathrm{rev}_1 \, L(\lambda)$ is also a linearization of $\mathrm{rev}_k \, P(\lambda).$
\end{definition}

\section{Vector spaces of generalized linearizations}

\noindent The vector spaces $\mathbb{L}_1(P)$ and $\mathbb{L}_2(P)$ defined by~\eqref{l1reg} and~\eqref{l2reg} were introduced in~\cite{MacMMM06a} as sources of linearizations for a given square regular matrix polynomial $P(\lambda).$ This work was subsequently extended in~\cite{DeDM09} to the case of square singular matrix polynomials. In this section we extend the notion of these spaces to the case of rectangular matrix polynomials. For this we introduce the notion of generalized linearizations of matrix polynomials which we refer to  as g-linearizations in short. We then define vector spaces of matrix pencils corresponding to the polynomial $P(\lambda)$ and show that they have properties with respect to g-linearizations that closely resemble those of $\L_1(P)$ and $\L_2(P)$ established in~\cite{DeDM09} with respect to linearizations in the square singular case.

\subsection{Generalized linearizations of matrix polynomials}

\begin{definition}[g-Linearization]
 A matrix pencil $L(\lambda)=\lambda X+Y$ with $X,Y\in \mathbb{C}^{mk\times nk} $ is called a g-linearization of an $m \times n$ matrix polynomial $P(\lambda)$ of grade $k$ if there exist two unimodular matrices $E(\lambda)\in\mathbb{F}[\lambda]^{mk\times mk}$ and $F(\lambda)\in \mathbb{F}[\lambda]^{nk\times nk}$ such that
 $$E(\lambda)L(\lambda)F(\lambda)=\left[ \begin{array}{cc} P(\lambda) & \\ & I_{k-1} \otimes I_{m,n} \end{array} \right]$$
Here $ I_{m,n} =\begin{bmatrix}I_n\\0_{(m-n) \times n}\end{bmatrix} \text{ if } m>n$, $ I_{m,n} = \begin{bmatrix}I_m & 0_{m \times (n-m)} \end{bmatrix} \text{ if } m<n$ and $I_{m,n}=I_m=I_n \text{ if } m=n.$
\end{definition}

\begin{definition}
 A matrix pencil $L(\lambda)=\lambda X+Y$ with $X,Y\in \mathbb{C}^{mk\times nk} $ is a strong g-linearization of an $m \times n$ matrix polynomial $P(\lambda)$ of grade $k$ if $L(\lambda)$ is a g-linearization of $P(\lambda)$ and $\mathrm{rev}_1L(\lambda)$ is a g-linearization of $\mathrm{rev}_kP(\lambda).$
\end{definition}

From the above definition, it is clear that every linearization of a square matrix polynomial is also a generalized linearization, which justifies our choice for the term. Also, evidently a matrix polynomial has the same eigenvalues and elementary divisors as it g-linearization and the same finite and infinite eigenvalues and elementary divisors as its strong g-linearization. Therefore, to establish that the solution of a complete eigenvalue problem for a rectangular matrix polynomial can be obtained from a given strong g-linearization, it is enough to show that the minimal bases and indices of the polynomial can be easily recovered from the g-linearization.

\subsection{The vector spaces $\mathbb{L}_1(P)$ and $\mathbb{L}_2(P)$}

To extend the work in~\cite{MacMMM06a} to non-square matrix polynomials, we propose the following vector spaces, which we continue to denote by $\L_1(P)$ and $\L_2(P)$ for ease of notation.

\begin{eqnarray}
\mathbb{L}_1(P) & := & \{L(\lambda):L(\lambda)(\Lambda_k(\lambda)\otimes I_n)=v\otimes P(\lambda),v\in\mathbb{F}^k\}, \label{l1g} \\
\mathbb{L}_2(P) & := & \{L(\lambda):(\Lambda_k(\lambda)^T\otimes I_m)L(\lambda)=w^T\otimes P(\lambda),w\in\mathbb{F}^k\}. \label{l2g}
\end{eqnarray}

Following~\cite{MacMMM06a} we will refer to the vector $v$ ($w$) in the identity in~\eqref{l1g}, (\eqref{l2g}) satisfied by $L(\lambda) \in \L_1(P),$ ($L(\lambda) \in \L_2(P)$) as the right (left) ansatz vector corresponding to $L(\lambda).$ The sets $\mathbb{L}_1(P)$ and $\mathbb{L}_2(P)$ are not empty as $ C_1^g(\lambda) := \lambda X_1+Y_1 \in \mathbb{L}_1(P)$ with right ansatz vector $v = e_1 \in \F^k$\\
where $X_1=\begin{bmatrix}A_k & & &\\ &I_{m,n}& & \\ & & \ddots &\\ & & & I_{m,n}\end{bmatrix}$, $Y_1=\begin{bmatrix}A_{k-1} &A_{k-2} &\dots &A_0\\ -I_{m,n}& & &0\\ & \ddots&  &\\ & & -I_{m,n}& 0\end{bmatrix}$.\\

\noindent and $C_2^g(\lambda) := \lambda X_2+Y_2\in \mathbb{L}_2(P)$ with left ansatz vector $w = e_1 \in \F^k,$\\
where $X_2=\begin{bmatrix}A_k & & &\\ &I_{m,n}& & \\ & & \ddots &\\ & & & I_{m,n}\end{bmatrix}$, $Y_2=\begin{bmatrix}A_{k-1} &-I_{m,n} &\dots & \\ A_{k-2}& & &\\ & \ddots&  &-I_{m,n}\\ A_0&0 & & 0\end{bmatrix}$.

As Theorem~\ref{Zthm} and Theorem~\ref{Zthm2} show, if $m\geq n$ then $C_1^g(\lambda)$ is a strong g-linerization of $P(\lambda)$ and if $m\leq n$ then $C_2^g(\lambda)$ is a strong g-linerization of $P(\lambda)$.

For any matrix polynomial $P(\lambda)$, clearly $\mathbb{L}_1(P)$ and  $\mathbb{L}_2(P)$ are vector spaces over $\F.$ In this section we find some important properties of these vector spaces. The results show that if the $m \times n$ matrix polynomial $P(\lambda)$ is tall, i.e., $m\geq n,$ then the properties of $\L_1(P)$ with respect to g-linearizations are very similar to those of the corresponding space for square matrix polynomials considered in \cite{MacMMM06a} and \cite{DeDM09} with respect to linearizations. The same is true of  $\L_2(P)$ when $P(\lambda)$ is broad, i.e., $m\leq n.$

For the case $m = n,$ the matrix pencils in $\L_1(P)$ and $\L_2(P)$ were originally characterized in \cite{MacMMM06a} by introducing special operations on block matrices called column shifted sums and row shifted sums respectively. We state these definitions with the aim of showing that the same characterizations also hold when $m \neq n.$

\begin{definition}[Column and row shifted sums]
Let $X$ and $Y$ be block matrices \\
$$X=\begin{bmatrix}X_{11}&\dots&X_{1k}\\X_{21}&\dots&X_{2k}\\ \vdots &\ddots&\vdots\\X_{k1}&\dots&X_{kk}\end{bmatrix},Y=\begin{bmatrix}Y_{11}&\dots&Y_{1k}\\Y_{21}&\dots&Y_{2k}\\ \vdots &\ddots&\vdots\\Y_{k1}&\dots&Y_{kk}\end{bmatrix}$$ with blocks $X_{ij}, Y_{ij}\in\mathbb{F}^{m\times n}$ then the operations \\

\begin{eqnarray*}
X~ \columnshiftsum~ Y & := & \begin{bmatrix}X_{11}&\cdots&X_{1k}&0\\X_{21}&\cdots&X_{2k}&0\\ \vdots &\vdots&\ddots&\vdots\\X_{k1}&\dots&X_{kk}&0\end{bmatrix}+\begin{bmatrix}0&Y_{11}&\cdots&Y_{1k}\\0&Y_{21}&\cdots&Y_{2k}\\ \vdots&\vdots &\ddots&\vdots\\0&Y_{k1}&\cdots&Y_{kk}\end{bmatrix}, and \\
 X~ \rowshiftsum~ Y & := & \begin{bmatrix}X_{11}& X_{12} & \dots &X_{1k} \\  \vdots &\vdots&\ddots&\vdots\\X_{k1}& X_{k2} & \cdots & X_{kk} \\ 0 & 0 & \cdots & 0\end{bmatrix}+\begin{bmatrix}0 & 0 & \cdots & 0 \\ Y_{11} & Y_{12} & \cdots & Y_{1k} \\ \vdots & \vdots & \ddots & \vdots \\ Y_{k1} & Y_{k2} & \cdots & Y_{kk}\end{bmatrix},
\end{eqnarray*}
where the zero blocks are also of size $m\times n$ are referred to as the column shifted sum and the row shifted sum of $X$ and $Y$ respectively.
\end{definition}

The above definition immediately gives the following lemma, the proof of which is obvious.
\begin{lemma}
 Let $P(\lambda)=\sum_{i=0}^k \lambda^i A_i$ be an $m\times n$ matrix polynomial of grade $k$ and $L(\lambda)=\lambda X+Y$ be an $km\times kn$ pencil. Then for $v, w \in \mathbb{F}^k$,
 \begin{eqnarray*}
 (\lambda X+Y)(\Lambda_k(\lambda)\otimes I_n) & = & v \otimes P(\lambda) \Leftrightarrow X~\columnshiftsum ~Y=v\otimes\begin{bmatrix}A_k&A_{k-1}&\dots &A_0\end{bmatrix} \\
 (\Lambda_k(\lambda)^T\otimes I_m)(\lambda X+Y) & = & w^T \otimes P(\lambda) \Leftrightarrow X~\rowshiftsum ~Y=w^T\otimes\begin{bmatrix}A_k^T & A_{k-1}^T & \dots & A_0^T\end{bmatrix}^T
 \end{eqnarray*}
\end{lemma}

Thus we have an immediate characterization of the spaces $\mathbb{L}_1(P)$ and $\L_2(P)$ in the next theorem the proof of which is omitted as it follows by arguing exactly as in the proof of ~\cite[Theorem 3.5]{MacMMM06a}.

\begin{theorem}\label{charL1}
 Let $P(\lambda)=\sum_{i=0}^k \lambda^i A_i$ be a $m\times n$ matrix polynomial of grade $k$ and $v,w \in \mathbb{F}^k.$ Then the pencils in $\mathbb{L}_1(P)$ with right ansatz vector $v$ consists of all $L(\lambda)=\lambda X+Y$ such that $X=\begin{bmatrix}v\otimes A_k & -W\end{bmatrix}$ and $Y=\begin{bmatrix}W+v\otimes \begin{bmatrix}A_{k-1}&\dots&A_1\end{bmatrix} & v\otimes A_0\end{bmatrix}$ with $W\in \mathbb{F}^{km\times (k-1)n}$ chosen arbitrarily.

 Similarly, the pencils in $\L_2(P)$ with left ansatz vector $w$ are given by $L(\lambda) = \lambda X + Y$ such that $X =  \begin{bmatrix} w^T \otimes A_k \\ -\hat W\end{bmatrix}$ and $Y=\begin{bmatrix} \hat W + w^T\otimes \begin{bmatrix}A_{k-1}^T & \dots & A_1^T\end{bmatrix}^T \\ w^T \otimes A_0\end{bmatrix}$ with $\hat W \in \mathbb{F}^{(k-1)m\times kn}$ chosen arbitrarily.
\end{theorem}

It is clear from Theorem~\ref{charL1} that the vector spaces $\mathbb{L}_1(P)$ and $\mathbb{L}_2(P)$ are completely determined by the pairs $(v, W)$ and $(w, \hat W)$ respectively,
 where $v, w \in \mathbb{F}^k$, $W \in \mathbb{F}^{km\times (k-1)n}$ and $\hat W \in \F^{(k-1)m \times kn}.$ Hence the dimensions of the vector spaces $\mathbb{L}_1(P)$ and $\L_2(P)$ over $\F$ are both equal to $k(k-1)mn+k$. The following immediate corollary of Theorem~\ref{charL1} shows that in particular matrix pencils in $\L_1(P)$ and $\L_2(P)$ with corresponding ansatz vector $\alpha e_1 \in \F^k$ for some non zero scalar $\alpha$ are easy to construct from the coefficient matrices of $P(\lambda).$

\begin{corollary}
 Suppose $L(\lambda)=\lambda X+Y\in \mathbb{L}_1(P)$ with right ansatz vector $v=\alpha e_1$ for $\alpha\neq0$. Then $X=\left[\begin{array}{c|c}\alpha A_k & X_{12}\\ \hline & -Z\end{array}\right]$ and $Y=\left[\begin{array}{c|c} Y_{11}& \alpha A_0\\ \hline Z & \end{array}\right]$ where $X_{12}, Y_{11} \in \F^{m \times (k-1)n}$ satisfy $X_{12} + Y_{11} =
 \alpha\begin{bmatrix}A_{k-1}&\dots&A_1\end{bmatrix}$ and $Z\in\mathbb{F}^{(k-1)m\times (k-1)n}$ is arbitrary.

 Similarly, if $L(\lambda) = \lambda X+Y \in \mathbb{L}_2(P)$ has left ansatz vector $w = \alpha e_1$ for $\alpha\neq0,$ then $X=\left[\begin{array}{c|c}\alpha A_k & \\ \hline \hat X_{12} & -\hat Z\end{array}\right]$ and $Y=\left[\begin{array}{c|c} \hat Y_{11} & \hat Z \\ \hline \alpha A_0 & \end{array}\right]$ where $\hat X_{12}, \hat Y_{11} \in \F^{(k-1)m \times n}$ satisfy $\hat X_{12} + \hat Y_{11} = \alpha \begin{bmatrix}A_{k-1} \\ \vdots \\ A_1 \end{bmatrix}$ and  $\hat Z \in\mathbb{F}^{(k-1)m \times (k-1)n}$ is arbitrary.
\end{corollary}

Given an $m \times n$ matrix polyomial $P(\lambda),$ it is easy to see that
\begin{equation}\label{l1tol2}
L(\lambda) \in \L_2(P) \Leftrightarrow L(\lambda)^T \in \L_1(P^T).
\end{equation}

Therefore, the results in the rest of the paper for $\L_1(P)$ where $P(\lambda)$ is of size $m \times n$ with $m \geq n,$ give rise to corresponding results for $\L_2(P)$ when $m \leq n$ with appropriate modifications. We provide proofs only for the statements concerning $\L_1(P)$ as the corresponding statements for $\L_2(P)$ follow either by using the correspondence~\eqref{l1tol2} or by similar independent arguments. The first among these is an analog of ~\cite[Theorem 4.1]{DeDM09}, that gives a sufficient condition for a pencil in $\L_1(P)$ (respectively, $\L_2(P)$) to be a strong g-linearization of $P(\lambda)$ when $m \geq n$ (respectively, $m \leq n$).

\begin{theorem}\label{Zthm}
 Let $P(\lambda)=\sum_{i=0}^k \lambda^i A_i$ be an $m\times n$ matrix polynomial. If $m \geq n$ and $L(\lambda)\in \mathbb{L}_1(P)$ with right ansatz vector $v \in \F^k \setminus \{ 0 \},$ then for any nonsingular $M\in\mathbb{F}^{k\times k}$ such that $Mv=\alpha e_1$ for some $\alpha \neq 0,$ the pencil $(M\otimes I_m)L(\lambda)$ satisfies
 \begin{equation}\label{Zrank} (M\otimes I_m)L(\lambda)=\lambda\left[\begin{array}{c|c}\alpha A_k & X_{12}\\ \hline & -Z\end{array}\right]+\left[\begin{array}{c|c} Y_{11}& \alpha A_0\\ \hline Z & \end{array}\right]\end{equation}
 with $Z\in\mathbb{F}^{(k-1)m\times (k-1)n}$. If $Z$ is of full rank, i.e., $\rank \, Z = (k-1)n,$ then $L(\lambda) \in \L_1(P)$  is a strong g-linearization of $P(\lambda)$.
\end{theorem}

\begin{proof}
We first prove the theorem for the case that $v = \alpha e_1$ for some $\alpha \neq 0.$ Then $$L(\lambda)=\lambda\left[\begin{array}{c|c}\alpha A_k & X_{12}\\ \hline & -Z\end{array}\right]+\left[\begin{array}{c|c} Y_{11}& \alpha A_0\\ \hline Z & \end{array}\right] = \lambda X + Y (say),$$ where $X_{12}, Y_{11} \in \F^{m \times (k-1)n}$ satisfy $X_{12} + Y_{11} =
 \alpha\begin{bmatrix}A_{k-1}&\dots&A_1\end{bmatrix}$ and $Z\in\mathbb{F}^{(k-1)m\times (k-1)n}$ is arbitrary.
 Partitioning $Z$ as
 $Z=\begin{bmatrix}Z_1& Z_2 & \dots & Z_{k-1}\end{bmatrix} \mbox{ where } Z_i \in \mathbb{F}^{(k-1)m \times n},$ and setting $$G(\lambda)=\begin{bmatrix} 1 & 0 & \dots & \lambda^{k-1}\\ & \ddots & &\vdots \\ & & 1 & \lambda \\ & & & 1\end{bmatrix}\otimes I_n,$$ we have,
 \begin{eqnarray*}
 L(\lambda)G(\lambda)&=&\begin{bmatrix}* & * & \dots& *& *\\Z_1 & (Z_2-\lambda Z_1) & \dots &(Z_{k-1}-\lambda Z_{k-2})& -\lambda Z_{k-1}\end{bmatrix}G(\lambda)\\
 &=&\begin{bmatrix}* & * & \dots& *& \alpha P(\lambda)\\Z_1 & (Z_2-\lambda Z_1) & \dots &(Z_{k-1}-\lambda Z_{k-2})& 0 \end{bmatrix}.
 \end{eqnarray*}
 Now,
 \begin{eqnarray*}
  &&L(\lambda)G(\lambda)
  \begin{bmatrix}I_n & \lambda I_n &  &  & \\ & I_n &  &  & \\ & &I_n &  &\\ & & &\ddots &\\&&&&I_n\end{bmatrix}
 \begin{bmatrix}I_n & &  &  & \\ & I_n & \lambda I_n &  & \\ & &I_n &  &\\ & & &\ddots &\\&&&&I_n\end{bmatrix} \dots
 \begin{bmatrix}I_n &  &  &  & \\ & \ddots &  &  & \\ & &I_n & \lambda I_n  &\\ & & &I_n&\\&&&&I_n\end{bmatrix}\\
&=& \begin{bmatrix}* & * & \dots& *& \alpha P(\lambda)\\Z_1 & Z_2 & \dots &(Z_{k-1}-\lambda Z_{k-2})& 0 \end{bmatrix}
 \begin{bmatrix}I_n & &  &  & \\ & I_n & \lambda I_n &  & \\ & &I_n &  &\\ & & &\ddots &\\&&&&I_n\end{bmatrix}\dots
 \begin{bmatrix}I_n &  &  &  & \\ & \ddots &  &  & \\ & &I_n & \lambda I_n  &\\ & & &I_n&\\&&&&I_n\end{bmatrix}\\
 &=&\begin{bmatrix}* & * & \dots& *& \alpha P(\lambda)\\Z_1 & Z_2 & \dots &Z_{k-1}& 0 \end{bmatrix}\\
 &=& \left[\begin{array}{c|c}* & \alpha P(\lambda)\\ \hline Z &  \end{array}\right].
 \end{eqnarray*}
Therefore there exist a unimodular matrix $F(\lambda)$ such that
\begin{equation}\label{F-stage}L(\lambda)F(\lambda)=\left[\begin{array}{c|c}  P(\lambda) & W(\lambda)\\ \hline  & Z \end{array}\right]\text{ for some $W(\lambda) \in \F[\lambda]^{m \times (k-1)n}.$}\end{equation}

\noindent If $Z$ is of full rank, then $Z^\dagger Z=I_{(k-1)n}.$ Therefore, $$\left[\begin{array}{c|c} I_m & -W(\lambda)Z^\dagger \\ \hline & I_{(k-1)m} \end{array}\right]L(\lambda)F(\lambda)=\left[\begin{array}{c c}  P(\lambda) & \\  & Z \end{array}\right].$$
As $\mathrm{rank} \, Z=(k-1)n=\mathrm{rank} \, (I_{k-1}\otimes I_{m,n})$, there exist invertible matrices
$E\in\mathbb{F}^{(k-1)m\times (k-1)m}$ and $F\in\mathbb{F}^{(k-1)n\times (k-1)n}$ such that $Z=E(I_{k-1}\otimes I_{m,n})F.$
This implies that $L(\lambda)$ is a g-linearization of $P(\lambda)$ as, $$\left[\begin{array}{c|c} I_m & -W(\lambda)Z^\dagger \\ \hline  & I_{(k-1)m} \end{array}\right]L(\lambda)F(\lambda)=\left[\begin{array}{c c}  I_m & \\   & E \end{array}\right]\left[\begin{array}{cc}  P(\lambda) & \\  & I_{k-1}\otimes I_{m,n} \end{array}\right]\left[\begin{array}{cc}  I_n & \\  & F \end{array}\right].$$
To show that $L(\lambda)$ is a strong g-linearization of $P(\lambda)$, notice that $$\lambda^{k-1}\Lambda_k(1/\lambda)=\begin{bmatrix}1& \lambda & \dots& \lambda^{k-1}\end{bmatrix}^T=R_k\Lambda_k(\lambda), \mbox{ where } R_k=\begin{bmatrix}&&1\\&\text{\reflectbox{$\ddots$}}&\\1&&\end{bmatrix}_{k \times k}.$$
As $L(\lambda) \in \L_1(P)$ with corresponding right ansatz vector $\alpha e_1,$ we have
$$\mathrm{rev}_1L(\lambda)(R_k\Lambda_k(\lambda)\otimes I_n)=\alpha e_1\otimes \mathrm{rev}_kP(\lambda)
 \Rightarrow \underbrace{\mathrm{rev}_1L(\lambda)(R_k \otimes I_n)}_{=: \tilde L(\lambda)}(\Lambda_k(\lambda) \otimes I_n) = \alpha e_1\otimes \mathrm{rev}_kP(\lambda).$$
Therefore $\tilde L(\lambda)=\lambda \tilde X+\tilde Y \in \L_1(\mathrm{rev} P)$, where $$\tilde X= Y(R_k\otimes I_n) = \left[\begin{array}{c|c}\alpha A_0&\tilde X_{12}\\ \hline &-\tilde Z\end{array}\right]  \text{ and } \tilde Y=X(R_k\otimes I_n)
= \left[\begin{array}{c|c}\tilde Y_{11}&\alpha A_k\\ \hline \tilde Z&\end{array}\right]$$
with $\tilde Z=-Z(R_{k-1}\otimes I_n)$ which is of full rank if $Z$ is of full rank. Hence $\tilde L(\lambda)$ is a g-linearization of $\mathrm{rev}_kP(\lambda)$ and consequently $\mathrm{rev}_1L(\lambda)$ is a g-linearization of $\mathrm{rev}_kP(\lambda)$, this completes the proof for the case that $v = \alpha e_1.$

Now let $L(\lambda) \in \L_1(P)$ with corresponding nonzero right ansatz vector $v \in \F^k.$ From \eqref{Zrank} it follows that $L(\lambda)$ is a strong g-linearization of $P(\lambda)$ if and only if the pencil $$\hat{L}(\lambda) := \lambda\left[\begin{array}{c|c}\alpha A_k & X_{12}\\ \hline & -Z\end{array}\right]+\left[\begin{array}{c|c} Y_{11}& \alpha A_0\\ \hline Z & \end{array}\right]$$ is a strong g-linearization of $P(\lambda).$ Clearly, $\hat{L}(\lambda) \in \L_1(P)$ with corresponding right ansatz vector $\alpha e_1.$ Since $\rank \, Z = (k-1)n,$ by the first part of the proof, it follows that $\hat{L}(\lambda)$ is a strong g-linearization of $P(\lambda)$ and this completes the proof.
\end{proof}

The corresponding theorem for $\L_2(P)$ is as follows.

\begin{theorem}\label{Zthm2}  Let $P(\lambda)=\sum_{i=0}^k \lambda^i A_i$ be an $m\times n$ matrix polynomial.
If, $m \leq n$ and $L(\lambda)\in \mathbb{L}_2(P)$ with left ansatz vector $w \in \F^k \setminus \{ 0 \},$  then for any nonsingular $\hat M \in\mathbb{F}^{k\times k}$ such that $\hat M w = \alpha e_1$ for some $\alpha \neq 0,$ the pencil $L(\lambda)(\hat M^T \otimes I_n)$ satisfies
\begin{equation}\label{hZrank}
L(\lambda)(\hat M^T \otimes I_n) =\left[\begin{array}{c|c}\alpha A_k & \\ \hline \hat X_{12} & -\hat Z\end{array}\right] + \left[\begin{array}{c|c} \hat Y_{11} & \hat Z \\ \hline \alpha A_0 & \end{array}\right],
\end{equation}
with $\hat Z \in \F^{(k-1)m \times (k-1)n}.$
   If $\hat Z$ is of full rank, i.e., $\rank \, \hat Z = (k-1)m,$ then $L(\lambda) \in \L_2(P)$  is a strong g-linearization of $P(\lambda)$.
\end{theorem}

It was proved in \cite[Theorem 4.1 and Theorem 4.3]{MacMMM06a} that if $P(\lambda)$ is a square regular polynomial, then $L(\lambda) \in \L_1(P)$ is a strong linearization of $P(\lambda)$ if and only if the matrix in the position of the block labelled $Z$ in \eqref{Zrank} is nonsingular. However as shown in \cite[Example 2]{DeDM09}, the same is not a necessary condition for $L(\lambda)$ to be a strong linearization of $P(\lambda)$ if it is square but not regular. The following simple modification of that example shows that if $P(\lambda)$ is an $m \times n$ matrix polynomial with $m\geq n,$ then $L(\lambda) \in \L_1(P)$ with corresponding nonzero right ansatz vector $v \in \F^k,$ can be a strong g-linearization of $P(\lambda)$
even if the matrix labelled $Z$ in \eqref{Zrank} is rank deficient.

\begin{example} \small{ Let $P(\lambda) = \lambda^2 A_2$ where $A_2 = \left[\begin{array}{cc} 1 & 0 \\ 0 & 0 \\ 0 & 0 \end{array}\right].$ Then $L(\lambda) = \lambda \left[\begin{array}{cc} A_2 & -\hat{X} \\ 0 & -Z \end{array}\right] + \left[\begin{array}{cc} \hat{X} & 0 \\ Z & 0 \end{array}\right]$ where $$\hat{X} = \left[\begin{array}{cc} 0 & 0 \\ 0 & -1 \\ 0 & 0 \end{array} \right], \quad Z = \left[\begin{array}{cc} -1 & 0 \\ 0 & 0 \\ 0 & 0 \end{array} \right]$$ belongs to $\L_1(P)$ with right ansatz vector $e_1.$ Although $\rank \, Z = 1,$ interchanging the second and fifth rows of $L(\lambda)$ gives $C_1^g(\lambda)$ which is a strong g-linearization of $P(\lambda).$}
\end{example}

Since the matrix in the block labelled $Z$ in the reduction \eqref{Zrank} of $L(\lambda) \in \L_1(P)$ plays an important role in determining whether $L(\lambda)$ is a g-linearization of $P(\lambda),$ we refer to it as the $Z$-matrix of $L(\lambda)$ with respect to the pair $(M, \alpha)$ as it may vary depending on the choice of the nonsingular matrix $M$ satisfying $Mv = \alpha e_1.$ Therefore it is important to know whether its rank can change with change in the choice of $M.$ The next theorem shows that this does not happen, i.e., the rank of the $Z$-matrix
in a given $L(\lambda) \in \L_1(P),$ remains invariant of the choice of $M.$ The proof of the theorem is omitted as it follows by arguing exactly as in the proof of \cite[Lemma 4.2]{DeDM09}.

\begin{theorem}
Let $P(\lambda)= \sum_{i=0}^k \lambda^i A_i$ be an $m\times n$  matrix polynomial with $m \geq n$ and $L(\lambda)=\lambda X+Y\in\mathbb{L}_1(P)$ with right ansatz vector $v\neq0$. Suppose that $M_1,M_2\in\mathbb{F}^{k\times k}$ are two nonsingular matrices such that $M_1v=\alpha_1e_1$ and $M_2v=\alpha_2e_1$ for some $\alpha_1\neq0,$ and $\alpha_2\neq0.$ If $Z_1,Z_2\in\mathbb{F}^{(k-1)m\times (k-1)n}$ are the matrices in the block labelled $Z$ in~\eqref{Zrank} corresponding to the pairs  $(M_1,\alpha_1)$ and $(M_2,\alpha_2)$ respectively, then $\rank \, Z_1=\rank \, Z_2$.
\end{theorem}

In a similar way it can also be shown that if the $m \times n$ matrix polynomial $P(\lambda)$ satisfies $m \leq n,$ the rank of the matrix labelled $\hat Z$ in the reduction~\eqref{hZrank} is independent of the choice of the nonsingular matrix $\hat M.$  The above result allows us to make the following definition.

\begin{definition}
  For an $m \times n$ matrix polynomial $P(\lambda)$ with $m \geq n,$ (respectively $m \leq n,$) the $Z$-rank of $L(\lambda)\in \mathbb{L}_1(P)$ (respectively, $L(\lambda) \in \L_2(P)$) is the rank of any matrix appearing in the block labelled $Z$ (respectively, $\hat Z$) under any reduction of $L(\lambda)$ of the form \eqref{Zrank} (respectively, \eqref{hZrank}). If $Z$ (respectively, $\hat Z$) in \eqref{Zrank} (respectively, \eqref{hZrank})  is of full rank, then we say that $L(\lambda)\in\mathbb{L}_1(P)$ (respectively, $L(\lambda)\in\mathbb{L}_2(P)$) has full $Z$-rank.
\end{definition}

The final result of this section shows that for a given $m \times n$ matrix polynomial $P(\lambda)$ with $m \geq n,$ almost every pencil in $\L_1(P)$ is a g-linearization of $P(\lambda).$


\begin{theorem}{\rm {\bf (Genericity of g-linearizations in $\L_1(P)$ and $\L_2(P)$)}}
 For any $m\times n$ matrix polynomial $P(\lambda)$ of grade $k$ with $m \geq n,$ (respectively, $m \leq n,$) almost every pencil in $\mathbb{L}_1(P)$ (respectively, $\L_2(P)$) is a strong g-linearization of $P(\lambda)$.
\end{theorem}

\begin{proof}
 Let $P(\lambda)=\sum_{i=0}^k \lambda^i A_i$ be an $m\times n$ matrix polynomial with $m\geq n.$
 The set of pencils in $\mathbb{L}_1(P)$ with right ansatz vector $v$ consists of all $L(\lambda)=\lambda X+Y$ such that
 $$X=\begin{bmatrix}v\otimes A_k & -W\end{bmatrix} \mbox{ and } Y=\begin{bmatrix}W+v\otimes \begin{bmatrix}A_{k-1}&\dots&A_1\end{bmatrix} & v\otimes A_0\end{bmatrix}$$ with $W\in \mathbb{F}^{km\times (k-1)n}$ chosen arbitrarily. For a parameterized $\mathbb{L}_1(P)$ we define the isomorphism

 \begin{eqnarray*}
 \Gamma:\mathbb{L}_1(P) &\rightarrow& \mathbb{F}^k\times\mathbb{F}^{km \times (k-1)n}\\
 \lambda X+Y &\mapsto& (v,W).
 \end{eqnarray*}

 \noindent Suppose $L(\lambda)=\lambda X+Y\in\mathbb{L}_1(P)$ with right ansatz vector $v=\begin{bmatrix}v_1&\dots&v_k\end{bmatrix}^T$.
 Let $$M=\left[\begin{array}{c|c}1&0\\ \hline \begin{matrix}-v_2\\ \vdots \\ -v_k\end{matrix}&v_1I_{k-1}\end{array}\right],$$ then $Mv=v_1e_1$ and if $v_1\neq 0$ then $M$ is nonsingular.
Now $(M\otimes I_m)L(\lambda)=\lambda ((M\otimes I_m)X)+((M\otimes I_m)Y)$ where
 $$(M\otimes I_m)X=\left[\begin{array}{c|c} Mv\otimes A_k & -(M\otimes I_m)W\end{array}\right]=\left[\begin{array}{c|c}v_1e_1\otimes A_k& -(M\otimes I_m)W\end{array}\right]:=\left[\begin{array}{c|c}v_1A_k&*\\ \hline &\tilde Z\end{array}\right].$$
 Clearly $-\tilde Z$ is a $Z$-matrix of $L(\lambda)$ with respect to $(M, v_1)$.
Then $$\mathcal{P}(v,W):=v_1\sum \det(\text{ minor of } \tilde Z \text{ of order } (k-1)n)$$ is a polynomial
in the $k+k(k-1)mn$ entries of $v$ and $W$. The pair corresponding to $C_1^g(\lambda)$ has
 $v=e_1$ and $W=\left[\begin{array}{c} 0\\ \hline -I_{k-1} \otimes I_{m,n} \end{array}\right].$
Hence, $\tilde Z= I_{k-1} \otimes I_{m,n}$ and thus $\mathcal{P}(v,W)\neq0$ for $C_1^g(\lambda)$.
 Therefore the zero set of $\mathcal{P}(v,W)$ defines a proper algebraic subset of $\mathbb{L}_1(P)$. Clearly any pair $(v,W)$ such that $\mathcal{P}(v,W)\neq0$ has $v_1\neq0$ and any one of the minors of $\tilde Z$ of order $(k-1)n$ has nonzero determinant. So the corresponding $L(\lambda)$ will have full $Z$-rank and hence is a strong g-linearization of $P(\lambda)$.
\end{proof}

An important difference between linearizations of regular and singular square matrix polynomials $P(\lambda)$ in the space $\L_1(P)$ is that while every linearization of $P(\lambda)$ in $\L_1(P)$ is also a strong linearization of $P(\lambda)$ when $P(\lambda)$ is a regular matrix polynomial~\cite[Theorem 4.3]{MacMMM06a}, the same is not true if $P(\lambda)$ is singular~\cite[Example 3]{DeDM09}. The following example shows that the same also holds for g-linearizations of rectangular matrix polynomials, i.e., there exist rectangular matrix polynomials $P(\lambda)$ with g-linearizations in $\L_1(P)$ that are not strong g-linearizations.

\begin{example}
\small{
Let $P(\lambda) = \left[\begin{array}{cc} \lambda^2 & \lambda \\ \lambda & 1 \\ 0 & 0 \end{array}\right].$ Then $$L(\lambda) = \lambda\left[\begin{array}{cccc} 1 & 0 & 0 & 0 \\
0 & 0 & 0 & 0 \\ 0 & 0 & 0 & 0 \\ 0 & 0 & -1 & 0 \\ 0 & 0 & 0 & 0 \\ 0 & 0 & 0 & 0 \end{array} \right] + \left[\begin{array}{cccc} 0 & 1 & 0 & 0 \\ 1 & 0 & 0 & 1 \\ 0 & 0 & 0 & 0 \\ 1 & 0 & 0 & 0\\ 0 & 0 & 0 & 0 \\ 0 & 0 & 0 & 0 \end{array}\right] \in \L_1(P)$$
is a g-linearization of $P(\lambda)$ as $E(\lambda)L(\lambda)F(\lambda) =  \left[\begin{array}{cc} P(\lambda) & 0 \\ 0 & I_{3,2} \end{array}\right]$ for
$$E(\lambda) = \left[\begin{array}{cccccc} 0 & 0 & 1 & \lambda & 0 & 0 \\ 0 & 0 & 0 & 1 & 0 & 0 \\ 0 & 0 & 0 & 0 & 1 & 0\\ 0 & 1 & 0 & 0 & 0 & 0 \\ 1 & 0 & 0 & 0 & 0 & 0 \\ 0 & 0 & 0 & 0 & 0 & 1 \end{array}\right] \mbox{ and } F(\lambda) = \left[\begin{array}{cccc} 0 & 1 & 0 & 0 \\ 0 & -\lambda & 0 & 1 \\ -1 & 0 & 0 & 0 \\ 0 & -1 & 1 & 0 \end{array}\right].$$
But $L(\lambda)$ is not a strong g-linearization of $P(\lambda)$ as infinity is a eigenvalue of $L(\lambda)$ but not of $P(\lambda)$.}
\end{example}

\section{Recovery of minimal indices and bases in $\L_1(P)$ and $\L_2(P)$}

In this section we show the process of extraction of left and right minimal bases and indices of an $m \times n$ polynomial $P(\lambda)$ from that of a g-linearization in $\L_1(P)$ or $\L_2(P).$ In particular we show that these extractions are possible from g-linearizations of $P(\lambda)$ in $\L_1(P)$ with full $Z$-rank if $m\geq n$ and those of $P(\lambda)$ in $\L_2(P)$ if $m\leq  n.$ It is easy to see that if $m\geq n$ and $L(\lambda) \in \L_1(P)$ is of full Z-rank, then
$$\dim \, N_r(L) = \dim \, N_r(P) \mbox{ and } \dim \, N_l (L) + (k-1)(m-n) = \dim \, N_l(P).$$
On the other hand if $m\leq n,$ and $L(\lambda) \in \L_2(P)$ has full Z-rank, then the above equalities hold when the positions of the right and left null spaces are interchanged for both $P(\lambda)$ and $L(\lambda).$ Therefore the process of extracting the right (respectively, left) minimal bases and indices of $P(\lambda) \in \F[\lambda]^{m \times n}$ from those of a g-linearization of $P(\lambda)$ in $\L_1(P)$ (respectively, $\L_2(P)$) is identical to the extraction of the same quantities from a linearization of a square singular polynomial in the respective spaces (as established in~\cite{DeDM09}). However, showing that the left (respectively, right) minimal bases and indices of $P(\lambda)$ can also be extracted from those of $L(\lambda) \in \L_1(P)$ (respectively, $L(\lambda) \in \L_2(P)$) with full $Z$-rank requires more work.

\subsection{Recovery of right (left) minimal indices and bases in $\L_1(P)$ ($\L_2(P)$)}

Given an $m \times n$ matrix polynomial $P(\lambda)$ with $m\geq n,$ the following lemma provides an isomorphism between $N_r(P)$ and $N_r(L)$ that enables extraction of the right minimal bases and indices of $P(\lambda)$ from those of $L(\lambda) \in \L_1(P).$

\begin{lemma}
 Let $P(\lambda)$ be an $m\times n$ matrix polynomial of grade $k$ with $m \geq n,$ $L(\lambda)\in\mathbb{L}_1(P)$ with nonzero right ansatz vector $v$, and $x(\lambda)\in\mathbb{F}(\lambda)^n$. Then $\Lambda_k(\lambda)\otimes x(\lambda)\in N_r(L)$ if and only if $x(\lambda)\in N_r(P)$. Moreover, if $L(\lambda)$ is a g-linearization of $P(\lambda)$, then the mapping
 \begin{eqnarray*}
 R_\Lambda:N_r(P)&\rightarrow& N_r(L)\\
 x(\lambda)&\mapsto& \Lambda_k(\lambda)\otimes x(\lambda).
 \end{eqnarray*}
 is a linear isomorphism between the $\mathbb{F}(\lambda)$-vector spaces $N_r(P)$ and $N_r(L).$ Furthermore, $x(\lambda)\in N_r(P)$ is a vector polynomial if and only if $\Lambda_k(\lambda)\otimes x(\lambda)\in N_r(L)$ is a vector polynomial.
\end{lemma}

We skip the proof as it follows by arguing exactly as in the proof of \cite[Lemma 5.1]{DeDM09}. Now the following theorem whose proof is immediate shows that the right minimal bases and indices of $P(\lambda) \in \F[\lambda]^{m \times n}, \, m\geq n,$ have a very simple relationship with those of a g-linearization $L(\lambda) \in \L_1(P)$ and can be easily extracted from the latter.

\begin{theorem}\label{rminbasis}
 Let $P(\lambda)$ be an $m\times n$ matrix polynomial of grade $k$ with $m\geq n$ and $\mathrm{nrank}P(\lambda)=r.$ Also let $L(\lambda)\in \mathbb{L}_1(P)$ be a g-linearization of $P(\lambda).$

 \begin{enumerate}
  \item The right minimal indices of $P(\lambda)$ are $\epsilon_1\leq\epsilon_2\leq\dots\leq\epsilon_{n-r}$ if and only if the right minimal indices of $L(\lambda)$ are $(k-1)+\epsilon_1\leq(k-1)+\epsilon_2\leq\dots\leq(k-1)+\epsilon_{n-r}$.
  \item Every right minimal basis of $L(\lambda)$ is of the form $\{\Lambda_k(\lambda)\otimes x_1(\lambda),\dots,\Lambda_k(\lambda)\otimes x_{n-r}(\lambda)\}$ where \\ $\{x_1(\lambda),\dots,x_{n-r}(\lambda)\}$ is a right minimal basis of $P(\lambda)$.
 \end{enumerate}

\end{theorem}

Similarly, if $P(\lambda)$ is an $m \times n$ matrix polynomial with $m\leq n,$ then the mapping
$$  R_\Lambda:N_l(P) \rightarrow N_l(L), \quad  y(\lambda) \mapsto \Lambda_k(\lambda)\otimes y(\lambda),$$
 is an isomorphism between $N_l(P)$ and $N_l(L)$ that also induces a bijection between vector polynomials in $N_l(P)$ and $N_l(L).$ This results in the following counterpart of Theorem~\ref{rminbasis} for extraction of the left minimal bases and indices of $P(\lambda)$ from those of $L(\lambda) \in \L_2(P)$ with full Z-rank.

\begin{theorem}\label{lminbasis2}
 Let $P(\lambda)$ be an $m\times n $ matrix polynomial of grade $k$ with $m \leq n$ and $\mathrm{nrank}P(\lambda)=r.$ Also let $L(\lambda)\in \mathbb{L}_2(P)$ be a g-linearization of $P(\lambda).$

 \begin{enumerate}
  \item The left minimal indices of $P(\lambda)$ are $\epsilon_1\leq\epsilon_2\leq\dots\leq\epsilon_{m-r}$  if and only if the left minimal indices of $L(\lambda)$ are $(k-1)+\epsilon_1\leq(k-1)+\epsilon_2\leq\dots\leq(k-1)+\epsilon_{m-r}$.
  \item Every left minimal basis of $L(\lambda)$ is of the form $\{\Lambda_k(\lambda)\otimes y_1(\lambda),\dots,\Lambda_k(\lambda)\otimes y_{m-r}(\lambda)\}$ where \\ $\{y_1(\lambda),\dots, y_{m-r}(\lambda)\}$ is a left minimal basis of $P(\lambda)$.
 \end{enumerate}

\end{theorem}

\subsection{Recovery of left (right) minimal indices and bases in $\L_1(P)$ ($\L_2(P)$)}

In this section we first show that the left minimal bases and indices of $P(\lambda) \in \F[\lambda]^{m \times n}$ with $m\geq n,$ can be extracted from the g-linearizations in $\L_1(P)$ that are of full Z-rank. The following lemmas will be very useful for establishing Theorem~\ref{lminbasis} which is the main result.

\begin{lemma}\label{mainl}
 Let $P(\lambda) = \sum_{i = 0}^k \lambda^iA_i$ be an $m\times n$ matrix polynomial of grade $k$ with $m \geq n.$ Suppose $L(\lambda)\in\mathbb{L}_1(P)$ has full $Z$-rank and right ansatz vector $v\neq 0$. Then the mapping
 \begin{eqnarray*}
 \mathcal{L}_v:N_l(L) &\rightarrow& N_l(P)\\
 y(\lambda) &\mapsto& (v^T\otimes I_m)y(\lambda)
 \end{eqnarray*}
 is a linear map from the vector space $N_l(L)$ onto the vector space $N_l(P)$ over $\F(\lambda).$ Furthermore it is an onto map from the vector polynomials in $N_l(L)$ to the vector polynomials in $N_l(P)$ with the property that if $q(\lambda) \in N_l(P)$ is a vector polynomial of degree $\delta,$ then there exists a vector polynomial $y(\lambda) \in N_l(L)$ of degree $\delta$ such that $\mathcal{L}_v(y(\lambda)) = q(\lambda).$
\end{lemma}

\begin{proof}
 Let $y(\lambda)\in N_l(L).$ Since $L(\lambda) \in \L_1(P),$
 $$ y(\lambda)^TL(\lambda)(\Lambda_k(\lambda)\otimes I_n)=0 \Rightarrow y(\lambda)^T(v\otimes P(\lambda))=0 \Rightarrow y(\lambda)^T(v\otimes I_m)P(\lambda)=0.$$
Therefore $(v^T\otimes I_m)y(\lambda)\in N_l(P)$ and this shows that $\mathcal{L}_v$ is well defined and clearly linear.

Let $q(\lambda) \in N_l(P).$ To show that $\mathcal{L}_v$ is onto we prove that there exists $y(\lambda)\in N_l(L)$ such that $\mathcal{L}_v(y(\lambda)) = q(\lambda).$
Let $$S_k(\lambda) := \begin{bmatrix}1&\lambda&\dots&\lambda^{k-2}\\&\ddots&\ddots&\\&&\ddots&\lambda\\&&&1\\&&&0\end{bmatrix}_{k\times (k-1)},$$
\begin{equation}\label{q1}
q_1(\lambda) :=\begin{bmatrix}q(\lambda)\\ \tilde q(\lambda)\end{bmatrix} \mbox{ and } y(\lambda)=(M^{T}\otimes I_m)q_1(\lambda)
\end{equation}
where \begin{equation}\label{tildeq}
\tilde q(\lambda)^T=-q(\lambda)^T\underbrace{(\lambda\begin{bmatrix}A_k & X_{12}\end{bmatrix}+\begin{bmatrix}Y_{11}& A_0\end{bmatrix})(S_k(\lambda)\otimes I_n)Z^\dagger}_{=: C}
\end{equation} and $M$ is a nonsingular matrix such that $Mv=e_1.$ Clearly $\mathcal L_v(y(\lambda))=q(\lambda)$.

Now $y(\lambda)\in N_l(L)$ if and only if $q_1(\lambda)\in N_l(\hat L)$, where $\hat{L}(\lambda)=(M\otimes I_m)L(\lambda).$ Also as $L(\lambda) \in \L_1(P)$ corresponds to right ansatz vector $v$ and $Mv = e_1,$ therefore $\hat{L}(\lambda) \in \L_1(P)$ corresponds to right ansatz vector $e_1.$
So, $$\hat{L}(\lambda) = \lambda\left[ \begin{array}{c|c} A_k&X_{12}\\ \hline & -Z\end{array}\right ]+\left [\begin{array}{c|c} Y_{11}& A_0\\ \hline Z & \end{array}\right ]$$
where $Z$ has full rank and $X_{12} + Y_{11} = \left[\begin{array}{cccc} A_{k-1} & A_{k-2} & \cdots & A_1 \end{array}\right].$
 This implies that, $$X_{12} := \begin{bmatrix}X_1&X_2&\cdots&X_{k-1}\end{bmatrix} \mbox{ and } Y_{11} := \begin{bmatrix}Y_1&Y_2&\cdots&Y_{k-1}\end{bmatrix}$$ where $X_i,Y_i\in\mathbb{F}^{m\times n}$ satisfy $X_i + Y_i = A_{k-i}$ for $i = 1, 2, \ldots k-1.$
Now from \eqref{q1} and \eqref{tildeq},
\begin{eqnarray}
q_1(\lambda)^T\hat{L}(\lambda) & = & [q(\lambda)^T \quad \tilde{q}(\lambda)^T]\hat{L}(\lambda) \nonumber \\
                               & = & q(\lambda)^T[I_m \quad -C]\left(\lambda\left[ \begin{array}{c|c} A_k&X_{12}\\ \hline & -Z\end{array}\right ]+\left [\begin{array}{c|c} Y_{11}& A_0\\ \hline Z & \end{array}\right ]\right) \nonumber \\
                               & = &\label{C1} q(\lambda)^T\left(\lambda[A_k \quad X_{12} + \hat{C}] + [Y_{11}-\hat{C} \quad A_0]\right)
                               \end{eqnarray}
where
\begin{eqnarray}
 \hat{C} &=&(\lambda\begin{bmatrix}A_k & X_1 & X_2 &\dots& X_{k-1}\end{bmatrix}+\begin{bmatrix}Y_1&Y_2&\dots&Y_{k-1} & A_0\end{bmatrix})(S_k \otimes I_n)\nonumber\\
 &=&(\begin{bmatrix}\lambda A_k+Y_1 & \lambda X_1+Y_2& \lambda X_2+Y_3&\dots& \lambda X_{k-2}+Y_{k-1}&\lambda X_{k-1}+A_0\end{bmatrix})(S_k \otimes I_n)\nonumber\\
 &=& [\begin{array}{ccc}\lambda A_k+Y_1 & \lambda^2A_k+\lambda (Y_1+X_1)+Y_2 & \lambda^3A_k+\lambda^2(Y_1+X_1)+\lambda (Y_2+X_2)+Y_3\end{array} \nonumber\\
 & &\begin{array}{ccc}\dots&\dots&\lambda^{k-1}A_k+\lambda^{k-2}(Y_1+X_1)+\dots+\lambda(Y_{k-2}+X_{k-2})+Y_{k-1}\end{array}] \nonumber \\
 & = & [\begin{array}{cccc}\lambda A_k+Y_1 & \lambda^2A_k+\lambda A_{k-1}+Y_2 & \dots&\lambda^{k-1}A_k+\lambda^{k-2}A_{k-1}+\dots+\lambda A_2+Y_{k-1}\end{array}] \nonumber \\
 & = &\label{C2} [\begin{array}{cccc}\lambda A_k & \lambda^2A_k+\lambda A_{k-1} & \dots&\lambda^{k-1}A_k+\lambda^{k-2}A_{k-1}+\dots+\lambda A_2\end{array}] + Y_{11}
 \end{eqnarray}
the $2$-nd last equality being due to the fact that $X_i + Y_i = A_{k-i}, i = 1, 2, \ldots, k-1.$ Therefore,
\begin{eqnarray}
\lambda(X_{12} + \hat{C}) & = & \lambda(X_{12} + Y_{11}) + \lambda[\begin{array}{cccc}\lambda A_k & \lambda^2A_k+\lambda A_{k-1} & \dots&\lambda^{k-1}A_k+\lambda^{k-2}A_{k-1}+\dots+\lambda A_2\end{array}] \nonumber \\
& = & [\begin{array}{cccc}\lambda^2 A_k + \lambda A_{k-1} & \lambda^3A_k+\lambda^2 A_{k-1} + \lambda A_{k-2} & \dots& \lambda^k A_k+\lambda^{k-1}A_{k-1}+\dots+\lambda^2 A_2 +\lambda A_1\end{array}] \nonumber \\
&  & \label{C3}.
\end{eqnarray}
Using \eqref{C3} and \eqref{C2} in \eqref{C1},
\begin{eqnarray*}
q_1(\lambda)^T\hat{L}(\lambda) & = & q(\lambda)^T\left(\left[\begin{array}{ccccc} \lambda A_k & \lambda^2A_k + \lambda A_{k-1} & \lambda^3A_k+ \lambda^2 A_{k-1} + \lambda A_{k-2} & \cdots & P(\lambda) - A_0 \end{array}\right]\right.  \\
&  & + \left. \left[\begin{array}{ccccc}  -\lambda A_k &-(\lambda^2A_k + \lambda A_{k-1}) & -(\lambda^3A_k+ \lambda^2 A_{k-1} + A_{k-2}) & \cdots & A_0 \end{array} \right]\right) \\
& = & q(\lambda)^T \left[\begin{array}{ccccc} 0 & 0 & \cdots & 0 & P(\lambda) \end{array}\right] \\
& = & 0.
\end{eqnarray*}

Therefore $q_1(\lambda) \in N_l(\hat{L})$ and hence $\mathcal{L}_v$ is an onto linear map from the vector space $N_l(L)$ to the vector space $N_l(P)$
over $\F(\lambda).$ Now clearly, if $y(\lambda) \in N_l(L)$ is a vector polynomial, then so is $\mathcal{L}_v(y(\lambda)).$ Conversely, if $q(\lambda) \in N_l(P)$ is a vector polynomial, then from \eqref{q1} and \eqref{tildeq} it follows that $\tilde q(\lambda),$ $q_1(\lambda)$ and $y(\lambda)$ are also a vector polynomials. Since  $\mathcal{L}_v(y(\lambda)) = q(\lambda)$ and $y(\lambda) \in N_l(L)$, it follows that $\mathcal{L}_v$ maps the vector polynomials in $N_l(L)$ onto the vector polynomials in $N_l(P).$ To complete the proof we show that if the degree of $q(\lambda)$ is $\delta,$ then $y(\lambda)$ can be chosen so that it has degree $\delta.$

 Let $q(\lambda) \in N_l(P)$ and $y(\lambda) \in N_1(L)$ be vector polynomials such that $\mathcal{L}_v(y(\lambda))=q(\lambda).$ Let $\deg \, q = \delta$
 and suppose $\mathrm{deg} \, y =\hat \delta>\delta.$ Let $q_1(\lambda) = (M^{-T} \otimes I_m)y(\lambda).$ Then $\mathrm{deg} \, q_1 = \hat{\delta}$
 and $$q(\lambda) = \mathcal{L}_v (y(\lambda)) = (v^T \otimes I_m)(M^T \otimes I_m)q_1(\lambda) = (e_1^T \otimes I_m)q_1(\lambda).$$
 This implies that $q_1(\lambda) = \left[\begin{array}{c} q(\lambda) \\ {\tilde q}(\lambda) \end{array} \right]$ where $\mathrm{deg} \, {\tilde q} = \hat{\delta}.$
 Hence $$q_1(\lambda) = \sum_{i = \delta +1}^{\hat{\delta}} \lambda^i \left[\begin{array}{c} 0 \\ t_i \end{array}\right] + \underbrace{\left[\begin{array}{c}
 q(\lambda) \\ \sum_{i = 0}^\delta \lambda^i t_i \end{array}\right]}_{=:\hat{q}(\lambda)}$$
 where $t_i \in \F^{(k-1)m}, i = 0, \ldots, \hat{\delta},$ with $t_{\hat{\delta}} \neq 0.$
 Clearly $\mathrm{deg} \, \hat{q} = \delta$ and
\begin{equation}\label{eta} q(\lambda) = \mathcal{L}_v(y(\lambda)) = (e_1^T \otimes I_m)q_1(\lambda) = (e_1^T \otimes I_m)\hat{q} (\lambda) = (v^T \otimes I_m)
\underbrace{(M^T \otimes I_m)\hat{q}(\lambda)}_{=: \eta( \lambda)} = \mathcal{L}_v(\eta(\lambda)). \end{equation}
Since $\mathrm{deg} \, \hat{q} = {\delta},$ it follows that $\mathrm{deg} \, \eta = {\delta}.$ To complete the proof we show
that $\eta(\lambda) \in N_l(L).$ Now $y(\lambda) \in N_l(L) \Rightarrow q_1(\lambda) \in N_l(\hat{L}).$ Therefore,
\begin{eqnarray}
& & q_1(\lambda)^T\hat{L}(\lambda) = 0 \nonumber \\
& \Rightarrow & \left( \left[\begin{array}{c} 0_m \\ \sum_{i = \delta + 1}^{\hat{\delta}} \lambda^i t_i \end{array} \right]^T + \hat{q}(\lambda)^T \right)
 \left(\lambda\left[\begin{array}{c|c} A_k  &  X_{12} \\\hline & -Z \end{array}\right] +  \left[\begin{array}{c|c} Y_{11}  &  A_0 \\ \hline Z & \end{array}\right]\right) = 0 \nonumber \\
& \Rightarrow & \lambda \left[\begin{array}{c} 0_m \\ \sum_{i = \delta + 1}^{\hat{\delta}} \lambda^i t_i \end{array} \right]^T \left[\begin{array}{c|c} A_k  &  X_{12} \\ \hline & -Z \end{array}\right] + \left[\begin{array}{c} 0_m \\ \sum_{i = \delta + 1}^{\hat{\delta}} \lambda^i t_i \end{array} \right]^T\left[\begin{array}{c|c} Y_{11}  &  A_0 \\\hline Z & \end{array}\right] + \hat{q}(\lambda)^T\hat{L}(\lambda) = 0 \\
& \Rightarrow & \left[\begin{array}{cc} {0}^T_n & -\sum_{i = \delta + 1}^{\hat{\delta}} \lambda^{i+1}t_i^TZ \end{array}\right] + \left[\begin{array}{cc} \sum_{i = \delta + 1}^{\hat{\delta}} \lambda^it_i^TZ & 0_n^T\end{array}\right] + \hat{q}(\lambda)^T\hat{L}(\lambda) = 0 \nonumber \\
 & \Rightarrow & \lambda^{\delta + 1}\left[ t_{\delta +1}^TZ \quad 0_n^T\right] + \lambda^{\delta + 2}\left(\left[ t_{\delta + 2}^TZ \quad 0_n^T\right] +
 \left[ 0_n^T \quad -t_{\delta + 1}^TZ\right]\right) + \cdots \nonumber \\
 & &\label{finaldeg} \cdots + \lambda^{\hat{\delta}}\left( \left[ t_{\hat{\delta}}^TZ \quad 0_n^T\right] + \left[ 0_n^T \quad -t_{\hat{\delta} - 1}^T Z\right]\right) + \lambda^{\hat{\delta} + 1}\left[0_n^T \quad -t_{\hat{\delta}}^TZ  \right] + \hat{q}(\lambda)^T\hat{L}(\lambda) = 0.
\end{eqnarray}
Since the degree of $\hat{q}(\lambda)^T \hat{L}(\lambda)$ is atmost $\delta + 1,$ equating the coefficients of $\lambda^i, i = \delta + 2, \ldots, \hat{\delta} + 1$ in \eqref{finaldeg} to $0,$ we have $t_i^TZ = 0$ for $i = \delta + 1, \ldots, \hat{\delta}.$ Therefore \eqref{finaldeg} implies that $\hat{q}(\lambda)^T\hat{L}(\lambda) = 0$ and this completes the proof as
$$\eta(\lambda)^TL(\lambda) = \hat{q}(\lambda)^T(M \otimes I_m)L(\lambda) = \hat{q}(\lambda)^T\hat{L}(\lambda) = 0.$$
\end{proof}

\begin{lemma}\label{sbasis} Let $P(\lambda) = \sum_{i=0}^k \lambda^iA_i$ be an $m \times n$ matrix polynomial of grade $k$ with $m \geq n$ and let $L(\lambda) \in \L_1(P)$ correspond to a nonzero right ansatz vector $v \in \F^k$ with full Z-rank. Let $r = \mathrm{nrank}(P(\lambda)),$ $c = (k-1)(m-n)$ and $\mathcal{L}_v(y(\lambda)) = (v^T \otimes I_m)y(\lambda)$ for all $y(\lambda) \in N_l(L).$ Then there exists a minimal basis of $N_l(L)$ of the form $$\{y_1(\lambda),\dots,y_{m-r}(\lambda),u_{m-r+1},\dots,u_{m-r+c}\},$$ where  $\{u_{m-r+1},\dots,u_{m-r+c}\} \subset \F^{km}$ is a basis of $N(\mathcal {L}_v).$
\end{lemma}

\begin{proof} Since $L(\lambda)$ is a g-linearization of $P(\lambda)$ and $m\geq n,$ therefore $\mathrm{dim} (N_l(L)) = \mathrm{dim} (N_l(P)) + (k-1)(m-n).$
Let $r = \mathrm{nrank} P(\lambda)$ and $\{ v_1(\lambda), \ldots, v_{m-r}(\lambda)\}$ be a minimal basis of $N_l(P)$ with $\mathrm{deg} \, v_j = \delta_j$ for $j = 1, \ldots, m-r.$
By Lemma~\ref{mainl}, there exist linearly independent vectors $y_1(\lambda), \ldots, y_{m-r}(\lambda) \in N_l(L)$ such that $$\mathcal{L}_v (y_j(\lambda)) = v_j(\lambda), \mbox{ and } \mathrm{deg} \, y_j = \delta_j.$$
Also from Lemma~\ref{mainl} we have, $\rank \, \mathcal{L}_v = \mathrm{dim} \, N_l(P),$ and hence $\mathrm{dim}(N(\mathcal{L}_v)) = (k-1)(m-n) = c.$
Let $M \in \F^{k \times k}$ be a nonsingular matrix such that $Mv = e_1$ and $\hat{L}(\lambda) := (M \otimes I_m)L(\lambda).$
 Since $L(\lambda)$ has full Z-rank, $\hat{L}(\lambda) = \lambda \left[\begin{array}{c|c} A_k & X_{12} \\\hline & -Z \end{array}\right] + \left[\begin{array}{c|c} Y_{11} & A_0 \\\hline Z & \end{array}\right],$ where $\rank \, Z = (k-1)n.$ We show that there exists a basis of $N(\mathcal{L}_v)$ consisting vectors of the form $(M^T \otimes I_m)\left[\begin{array}{c} 0 \\ v \end{array}\right]$ such that $v^TZ = 0.$ Since $\rank \, Z = (k-1)n,$ there exist $c$ linearly independent vectors $v_1, \ldots, v_c \in \F^{(k-1)m}$ such that $v_i^TZ = 0$ for all $i = 1, \ldots, c.$ Therefore $\beta_v := \left\{ (M^T \otimes I_m)\left[\begin{array}{c} 0 \\ v_1 \end{array}\right], \ldots, (M^T \otimes I_m)\left[\begin{array}{c} 0 \\ v_c \end{array}\right]\right\}$ is a linearly independent subset of $N(\mathcal{L}_v)$ as
\begin{eqnarray*}
& & \left((M^T \otimes I_m)\left[\begin{array}{c} 0 \\ v_i \end{array}\right]\right)^TL(\lambda) = \left[\begin{array}{c} 0 \\ v_i \end{array}\right]^T(M \otimes I_m)L(\lambda) = \left[\begin{array}{c} 0 \\ v_i \end{array}\right]^T\hat{L}(\lambda) = v_i^TZ = 0 \\
& \text{and} & (v^T \otimes I_m)(M^T \otimes I_m)\left[\begin{array}{c} 0 \\ v_i \end{array}\right] =
(e_1^T \otimes I_m)\left[\begin{array}{c} 0 \\ v_i \end{array}\right] = 0
\end{eqnarray*}
for all $i = 1, \ldots, c.$ Clearly $\beta_v$ is also a basis of $N(\mathcal{L}_v)$ as it has $c$ linearly independent vectors. Let $$\beta = \{y_1(\lambda), \ldots, y_{m-r}(\lambda), u_{m-r+1}, \ldots, u_{m-r+c}\}$$ where $u_{m-r+j} :=  (M^T \otimes I_m)\left[\begin{array}{c} 0 \\ v_j \end{array}\right],$ for $j = 1, \ldots, c.$ Since $\rank \, \mathcal{L}_v = \mathrm{dim}(N_l(P)) = m-r,$ $\beta$ has $\mathrm{dim}(N_l(L))$ vectors. Therefore
 $\beta$ is a basis of $N_l(L)$ if it is a linearly independent set. Suppose there exist $a_1(\lambda), \ldots, a_{m-r+c}(\lambda) \in \F(\lambda)$ such that
\begin{equation}\label{minbasis}
a_1(\lambda)y_1(\lambda) + \cdots + a_{m-r}(\lambda)y_{m-r}(\lambda) + a_{m-r+1}(\lambda)u_{m-r+1} + \cdots + a_{m-r+c}(\lambda)u_{m-r+c} = 0.
\end{equation}
Then, $\mathcal{L}_v(a_1(\lambda)y_1(\lambda) + \cdots + a_{m-r}(\lambda)y_{m-r}(\lambda)) = 0 \Rightarrow a_1(\lambda)v_1(\lambda) + \cdots + a_{m-r}(\lambda)v_{m-r}(\lambda) =0.$ \\
This gives $a_j(\lambda) = 0$ for $j = 1, \ldots, m-r$ as $\{v_1(\lambda), \ldots, v_{m-r}(\lambda)\}$ is a basis of $N_l(P).$ So from \eqref{minbasis} we have $a_{m-r+1}(\lambda)u_{m-r+1} + \cdots + a_{m-r+c}(\lambda)u_{m-r+c} = 0$ which implies that $a_j(\lambda) = 0$ for $j = m-r+1, \ldots, m-r+c$ as $\beta_v$ is a basis of $N(\mathcal{L}_v).$  Hence $\beta$ is a basis of $N_l(L).$  Suppose that it is not a minimal basis of $N_l(L).$ Since the sum of the degrees of the polynomials in $\beta$ is $\sum_{j=1}^{m-r} \delta_j,$ there exists a minimal basis
$\hat{\beta} := \{ \hat{y}_1(\lambda), \ldots, \hat{y}_{m-r+c}(\lambda) \}$ of $N_l(L)$ such that $\sum_{i = 1}^{m-r+c} \mathrm{deg} \, \hat{y}_i < \sum_{i=1}^{m-r}\delta_i.$ Then $\mathcal{L}_v(\hat{\beta})$ is a spanning set in $N_l(P).$ Let $\{\mathcal{L}_v(\hat{y}_{i_1}(\lambda)), \ldots, \mathcal{L}_v(\hat{y}_{i_{m-r}}(\lambda))\} \subset \mathcal{L}_v(\hat{\beta})$ be a basis of $N_l(P).$ Then $$\sum_{j = 1}^{m-r} \mathrm{deg} \, \mathcal{L}_v(\hat{y}_{i_j}) < \sum_{j=1}^{m-r} \delta_j = \sum_{j=1}^{m-r}\mathrm{deg} \, y_j =\sum_{j=1}^{m-r}\mathrm{deg} \, v_j.$$
But this contradicts the assumption that $\{ v_1(\lambda), \ldots, v_{m-r}(\lambda)\}$ is a minimal basis of $N_l(P).$ Hence the proof.
\end{proof}

The following theorem now shows how the left minimal indices and bases of an $m \times n$ matrix polynomial $P(\lambda)$ with $m \geq n$ can be extracted from those of $L(\lambda) \in \L_1(P)$ with full Z-rank.

\begin{theorem}\label{lminbasis} Let $P(\lambda)$ be an $m \times n$ matrix polynomial of grade $k$ with $m \geq n,$ and let $L(\lambda) \in \L_1(P)$ corresponding to nonzero right ansatz vector $v \in \F^k$ be of full Z-rank. Let $r = \mathrm{nrank} \, P(\lambda),$ $c = (k-1)(m-n)$ and $\mathcal{L}_v(y(\lambda)) = (v^T \otimes I_m)y(\lambda)$ for all $y(\lambda) \in N_l(L).$
 If $\beta=\{y_1(\lambda),\dots,y_{m-r}(\lambda),u_{m-r+1},\dots,u_{m-r+c}\}$ be a minimal basis of $N_l(L)$ satisfying the properties of Lemma~\ref{sbasis}, then $\{\mathcal L_v(y_1(\lambda)),\dots,\mathcal L_v(y_{m-r}(\lambda))\}$ is a minimal basis of $N_l(P)$. Moreover if $\delta_1\geq \delta_2\geq\dots\geq \delta_{m-r}\geq 0=\underbrace{0=\dots=0}_{c \,\, \text{zeros}}$ are the left minimal indices of $L(\lambda)$ then $\delta_1\geq \delta_2\geq\dots\geq \delta_{m-r}$ are the left minimal indices of $P(\lambda)$.
\end{theorem}

\begin{proof}
Since $\{u_{m-r+1},\ldots,u_{m-r+c}\}$ forms a basis of $N(\mathcal{L}_v)$, hence $\mathcal{L}_v(u_i) = 0$ for $i=m-r+1,\ldots, m-r+c$.
Let $z_i(\lambda) = \mathcal{L}_v(y_i(\lambda))$ for $i=1, \ldots, m-r.$ Evidently, $\mathrm{deg} \, y_i = \mathrm{deg} \, z_i$  for all $i=1,\dots,m-r$
and $\{ z_1(\lambda), \ldots, z_{m-r}(\lambda)\}$ is a basis of $N_l(P).$ Now if it is not a minimal basis of $N_l(P),$ then there exists a basis $\{ \hat{z}_1(\lambda), \ldots, \hat{z}_{m-r}(\lambda)\}$ of $N_l(P)$ such that $\mathrm{deg} \, \hat{z}_{j_0} < \mathrm{deg} \, z_{j_0}$ for some $1 < j_0 < m-r.$
Consequently, by Lemma~\ref{mainl} and Lemma~\ref{sbasis}, there exists
a basis $$\hat{\beta} := \left\{ \hat{y}_1(\lambda), \ldots, \hat{y}_{m-r}(\lambda), u_{m-r+1}, \ldots, u_{m-r+c}\right\}$$ of $N_l(L)$ such that $\mathcal{L}_v(\hat{y}_j(\lambda)) = \hat{z}_j(\lambda)$ and $\mathrm{deg} \, \hat{y}_j = \mathrm{deg} \, \hat{z}_j$ for $j = 1, \ldots, m-r.$  The sum of the degrees of the vector polynomials in $\hat{\beta}$ are clearly lower than that of the ones in $\beta$ as, $\mathrm{deg} \, \hat{z}_{j_0} < \mathrm{deg} \, z_{j_0}.$ But this contradicts the fact that $\beta$ is a minimal basis of $N_l(L).$ Hence the proof follows. \end{proof}

\begin{remark}\label{conslminbasis}
A minimal basis of $N_l(P)$ may be extracted from a basis of $N_l(L)$ that satisfies the assumptions of Lemma~\ref{sbasis}. We outline the steps for constructing such a basis from \emph{any} given minimal basis of $N_l(L).$

\begin{enumerate}
 \item Let $\{y_1(\lambda),\dots,y_t(\lambda),y_{t+1}(\lambda),\dots, y_{m-r+c}(\lambda)\}$ be a minimal basis of $N_l(L)$ with $\mathrm{deg} \, y_{i} = d_i$ such that $d_1 \geq d_2 \geq \cdots \geq d_{m-r+c}.$ Without loss of generality we may assume that there exists $1 \leq t \leq m-r$ such that $d_i > 0$ for $i=1,\dots,t$ and $d_i = 0$ for $ i = t+1, \ldots, m-r+c.$

 \item Consider $u_{m-r+j}=(M^T\otimes I_m)\left[\begin{array}{c} 0 \\ v_j \end{array}\right]$, for $j=1,\dots,c$ such that $\{v_1,\dots,v_c\}$ forms a basis of left null space of $Z$. Such a basis may be obtained by choosing $v_1, \ldots, v_c$ to be the complex conjugate transpose of the last $c$ columns of the matrix $Q$ of a $QR$ decomposition of $Z.$ Then $u_{m-r+j}$ belongs to $N_l(L)$ and $\mathcal{L}_v(u_{m-r+j})=0$ for $j=1,\dots,c$.

 \item Clearly the set $\beta=\{y_1(\lambda),\dots,y_t(\lambda),u_{m-r+1},\dots,u_{m-r+c}\}$ is linearly independent. Check if  $y_{t+1}(\lambda)$ belongs to $\mathrm{span}(\beta),$ and include it in $\beta$ if this is not the case. Repeat the process for $i = t+2, \ldots, m-r+c$ with respect to the updated $\beta$ after each step.
\end{enumerate}
\end{remark}

\begin{remark}\label{genericlminbasis} There are situations when left minimal bases of $L(\lambda) \in \L_1(P)$
will generically satisfy the assumptions in Lemma~\ref{sbasis}. For example if  $n < m \leq (k+1)n,$ then generically, $\mathcal{A} := \left[\begin{array}{ccc} A_0 & \cdots & A_k \end{array}\right]$ is a full rank matrix and consequently, none of the left minimal indices of $P(\lambda)$ are zero. Consequently, there does not exist any vector polynomial of degree zero in a left minimal basis of $L(\lambda)$ that does not belong to $N(\mathcal{L}_v).$ This implies that any such basis must satisfy the assumptions of Lemma~\ref{sbasis}.
\end{remark}

The preceding results imply that if $P(\lambda) \in \F[\lambda]^{m \times n}$ with $m\leq  n,$ then the right minimal bases and indices of $P(\lambda)$ can be extracted from those of $L(\lambda) \in \L_2(P)$ of full Z-rank. In particular we have in this case the following counterpart of Lemma~\ref{mainl}  which can either be proved by arguing as in the proof of Lemma~\ref{mainl} or by using the relation~\eqref{l1tol2}.

\begin{lemma}
 Let $P(\lambda) = \sum_{i = 0}^k \lambda^iA_i$ be an $m\times n$ matrix polynomial of grade $k$ with $m \leq n.$ Suppose $L(\lambda)\in\mathbb{L}_2(P)$ has full $Z$-rank and nonzero left ansatz vector $w \in \F^k$. Then the mapping
 \begin{eqnarray*}
 \mathcal{L}_w:N_r(L) &\rightarrow& N_r(P)\\
 x(\lambda) &\mapsto& (w^T\otimes I_n)x(\lambda)
 \end{eqnarray*}
 is a linear map from the vector space $N_r(L)$ onto the vector space $N_r(P)$ over $\F(\lambda).$ Furthermore it is an onto map from the vector polynomials in $N_r(L)$ to the vector polynomials in $N_r(P)$ with the property that if $q(\lambda) \in N_r(P)$ is a vector polynomial of degree $\delta,$ then there exists a vector polynomial $x(\lambda) \in N_r(L)$ of degree $\delta$ such that $\mathcal{L}_w(x(\lambda)) = q(\lambda).$
\end{lemma}

Therefore, by arguing as in the proof of Lemma~\ref{sbasis}, there exists a minimal basis of $N_r(L)$ of the form $$\{x_1(\lambda), \dots, x_{n-r}(\lambda), u_{n-r+1}, \dots, u_{n-r+c}\},$$ where  $c = (n-m)(k-1)$ and $\{u_{n-r+1},\dots,u_{n-r+c}\} \subset \F^{kn}$ is a basis of $N(\mathcal {L}_w).$ This leads to the following theorem for extracting the right minimal indices and bases of $P(\lambda)$ from those of $L(\lambda) \in \L_2(P)$ with full Z-rank.

\begin{theorem}\label{rminbasis2} Let $P(\lambda)$ be an $m \times n$ matrix polynomial of grade $k$ with $m\leq n,$ and $L(\lambda) \in \L_2(P)$ with corresponding nonzero left ansatz vector $w \in \F^k$ be of full Z-rank. Let $r = \mathrm{nrank} \, P(\lambda),$ $c = (k-1)(n-m)$ and $\mathcal{L}_w(x(\lambda)) = (w^T \otimes I_n)x(\lambda)$ for all $x(\lambda) \in N_r(L).$
 If $\beta=\{x_1(\lambda),\dots, x_{n-r}(\lambda),u_{n-r+1},\dots,u_{n-r+c}\}$ be a minimal basis of $N_r(L)$ such that $\{u_{n-r+1},\dots,u_{n-r+c}\} \subset \F^{kn}$ is a basis of $N(\mathcal {L}_w),$ then the set $\{\mathcal L_w(x_1(\lambda)),\dots,\mathcal L_w(x_{n-r}(\lambda))\}$ is a minimal basis of $N_r(P)$. Moreover if $\delta_1\geq \delta_2\geq\dots\geq \delta_{n-r}\geq 0=\underbrace{0=\dots=0}_{c \,\, \text{zeros}}$ are the right minimal indices of $L(\lambda)$ then $\delta_1\geq \delta_2\geq\dots\geq \delta_{n-r}$ are the right minimal indices of $P(\lambda)$.
\end{theorem}

\section{Linearizations arising from g-linearizations}\label{lin}

Let $P(\lambda) = \sum_{i=0}^k \lambda^i A_i$ be an $m \times n$ matrix polynomial of grade $k$. In this section we show that although the pencils in $\L_1(P)$ and $\L_2(P)$ are generically g-linearizations of $P(\lambda),$ they can give rise to smaller pencils that are linearizations of $P(\lambda)$ from which the left and right minimal bases and indices of $P(\lambda)$ may be easily extracted. In the following we first describe the process of extracting these smaller pencils from g-linearizations of $P(\lambda)$ of full Z-rank in $\L_1(P)$ when $m\geq n.$

Let $L(\lambda) \in \L_1(P)$ with nonzero right ansatz vector $v \in \F^k$ be of full $Z$-rank. Let $M \in \F^{k \times k}$ be a nonsingular matrix such that $Mv = \alpha e_1$. From~(\ref{F-stage}), $(M \otimes I_m)L(\lambda)=\left[\begin{array}{c|c} P(\lambda) & W(\lambda) \\\hline  & Z \end{array}\right](F(\lambda))^{-1}$. Let
\begin{equation}\label{qr}
Z = Q\begin{bmatrix} \tilde R\\0\end{bmatrix}
 \end{equation}
 be a  $QR$ decomposition of  $Z$ where $Q \in \F^{(k-1)m \times (k-1)m}$ is an unitary matrix and $\tilde R \in \F^{(k-1)n \times (k-1)n}$ nonsingular and upper triangular. Then we have,

\begin{center}
\begin{equation}\label{extract_lin}
\begin{bmatrix}I_m&\\&Q^{*}\end{bmatrix}(M\otimes I_m)L(\lambda)=\left[\begin{array}{c|c} P(\lambda) & W(\lambda) \\\hline  & \tilde R \\ & 0 \end{array}\right]( F(\lambda))^{-1},
 \end{equation}
 \end{center}

Let $Q = \left[\begin{array}{cc} Q_1 & Q_2 \end{array}\right],$ be a partition of $Q$ such that $Z = Q_1\tilde{R}$ is the condensed $QR$ decomposition of $Z.$ Then recalling that $c = (k-1)(m-n),$ the submatrix formed by the last $c$ rows of the matrix $\begin{bmatrix}I_m&\\&Q^{*}\end{bmatrix}(M\otimes I_m) L(\lambda)$ is $\begin{bmatrix}0&Q_2^{*}\end{bmatrix}(M\otimes I_m) L(\lambda)$. Since the last $c$ rows of the matrix on the RHS of ~\eqref{extract_lin} are zero, we have  $\begin{bmatrix}0&Q_2^{*}\end{bmatrix}(M\otimes I_m) L(\lambda)=0$. Consider $D \in\mathbb{F}^{m+(k-1)n\times km}$ such that \begin{equation}\label{choiceofD}\begin{bmatrix} D \\ \begin{bmatrix}0&Q_2^{*}\end{bmatrix}(M\otimes I_m) \end{bmatrix}\end{equation} is nonsingular. Then $\begin{bmatrix} D \\ \begin{bmatrix}0&Q_2^{*}\end{bmatrix}(M\otimes I_m) \end{bmatrix}L(\lambda)=\begin{bmatrix}DL(\lambda)\\0\end{bmatrix}$. We set,

\begin{center}
\begin{equation}\label{tlin}
L_t(\lambda)=DL(\lambda).
\end{equation}
\end{center}

Given a g-linearization $L(\lambda) \in \L_1(P)$ of full $Z$-rank, the above process of extracting the pencil $L_t(\lambda)$ from $L(\lambda)$  clearly depends not only on $L(\lambda)$ but also on the choice of the nonsingular matrix $M \in \F^{k \times k}$ satisfying $Mv = \alpha e_1$ and the matrix $D\in \F^{m+(k-1)n \times km}$ such that the matrix in \eqref{choiceofD} is nonsingular. For ease of expression, We will refer to $L_t(\lambda)$ as \emph{ the trimmed version of $L(\lambda) \in \L_1(P),$ with respect to $M$ and $D$}, the sizes of the matrices $M$ and $D$ being evident from the context.

Clearly, for a given choice of nonsingular $M \in \F^{k \times k},$ there are infinitely many choices of $D$ in~\eqref{tlin}. One possible choice is to set $D$ to be the first $m + (k-1)n$ rows of $\begin{bmatrix}I_m&\\&Q^{*}\end{bmatrix}(M\otimes I_m).$ Then the corresponding linearization is
\begin{equation}\label{hatlt}
\hat L_t(\lambda) := \lambda\left[\begin{array}{c|c}\alpha A_k & X_{12}\\ \hline & -\tilde{R}\end{array}\right]+\left[\begin{array}{c|c} Y_{11}& \alpha A_0\\ \hline \tilde{R} & \end{array}\right].
\end{equation}
If $L(\lambda) = C_1^g(\lambda),$ then such a choice of $D$ results in $L_t(\lambda)$ being the first Frobenius companion linearization $C_1(\lambda).$  Every other linearization $L_t(\lambda)$ that is not of the form~\eqref{hatlt} is strictly equivalent to some $\hat L_t(\lambda)$ as

\begin{eqnarray} \nonumber
 L_t(\lambda) &=& DL(\lambda)\\ \nonumber
             &=& D(M^{-1}\otimes I_m)\left\{\lambda\left[\begin{array}{c|c}\alpha A_k & X_{12}\\ \hline & -Z\end{array}\right]+\left[\begin{array}{c|c} Y_{11}& \alpha A_0\\ \hline Z & \end{array}\right]\right\}\\ \nonumber
             &=& D(M^{-1}\otimes I_m)\begin{bmatrix}I_m&\\&Q\end{bmatrix}\left\{\lambda\left[\begin{array}{c|c}\alpha A_k & X_{12}\\ \hline & -\tilde R\\&0\end{array}\right]+\left[\begin{array}{c|c} Y_{11}& \alpha A_0\\ \hline \tilde R & \\0&\end{array}\right]\right\}\\ \nonumber
             &=& D(\underbrace{M^{-1}\otimes I_m)\begin{bmatrix}I_m&\\&Q_1\end{bmatrix}}_{= E_1}\left\{\lambda\left[\begin{array}{c|c}\alpha A_k & X_{12}\\ \hline & -\tilde R\end{array}\right]+\left[\begin{array}{c|c} Y_{11}& \alpha A_0\\ \hline \tilde R &\end{array}\right]\right\}\\
             &=& \underbrace{DE_1}_{=:\tilde D}\hat L_t(\lambda). \label{lt_hatlt}
\end{eqnarray}
where clearly, $\tilde{D} \in \F^{m+(k-1)n \times m+(k-1)n}$ is nonsingular as it satisfies
$$\left[\begin{array}{c|c} {\tilde D} & * \\ \hline & I_{(k-1)(m-n)}\end{array}\right] = \begin{bmatrix}D \\ \begin{bmatrix}0&Q_2^{*}\end{bmatrix}(M\otimes I_m) \end{bmatrix}(M^{-1}\otimes I_m)\begin{bmatrix}I_m&\\&Q\end{bmatrix}.$$

\begin{remark} The QR decomposition of the Z-matrix of the g-linearization $L(\lambda)$ that has been used to extract the pencils $L_t(\lambda),$ from $L(\lambda)$ can easily be replaced by any other decomposition like the rank revealing QR decomposition or the SVD of $Z$ without affecting the results and the analysis concerning these pencils. Therefore, the upper triangular structure of the block $\tilde{R}$ in~\eqref{hatlt} is not essential for the rest of the paper.
\end{remark}

\subsection{Trimming a g-linearization results in a strong linearization}

 We now show that trimming a g-linearization of full Z-rank results in a strong linearization of $P(\lambda)$ from which the left and right minimal bases and indices of $P(\lambda)$ can easily be recovered. In doing so, we establish the connection between the resulting pencils with some of the important classes of linearizations for rectangular matrix polynomials that have been recently introduced in the literature. We begin with the block minimal bases pencils introduced in~\cite{DopLPV16}.

\begin{definition} A block minimal bases pencil is a pencil of the form $$\left[\begin{array}{c|c} A(\lambda) & {\hat B}(\lambda)^T \\\hline B(\lambda) &  \end{array}\right], \left[\begin{array}{c} A(\lambda) \\ \hline B(\lambda) \end{array}\right] \mbox{ or } \left[\begin{array}{c|c} A(\lambda) & {\hat B}(\lambda)^T \end{array}\right]$$
where the rows of $B(\lambda)$ and $\hat B(\lambda)$ form minimal bases of the rational subspaces spanned by them.
\end{definition}

We will need a few important concepts and results related to block minimal bases pencils from~\cite{DopLPV16}. For convenience, following~\cite{DopLPV16},  we refer to a matrix polynomial whose rows form a minimal basis of the rational subspace spanned by them as a minimal basis. Such a minimal basis can be associated with a dual minimal basis defined as follows.

\begin{definition} A pair of minimal bases $B(\lambda) \in \F[\lambda]^{n_1 \times n}$ and $C(\lambda) \in \F[\lambda]^{n_2 \times n}$ are called dual minimal bases if $n_1 + n_2 = n$ and $B(\lambda)C(\lambda)^T = 0.$
\end{definition}

For example, the matrix polynomials
\begin{equation}\label{dualminb2}
  H_j(\lambda)=\left[\begin{array}{ccccc} -1 &\lambda  & & & \\  & -1 &\lambda  & & \\ & &\ddots & \ddots &  \\ & & &-1 & \lambda\end{array}\right]_{j \times (j+1)}.
\end{equation}
and $\Lambda_{j+1}(\lambda)^T$ given by~\eqref{dualminb1} are dual minimal bases. For most practical purposes, we will need the following special kind of block minimal bases from~\cite{DopLPV16}.

\begin{definition}
A block minimal bases pencil \begin{equation}\label{minb} L(\lambda) = \left[\begin{array}{c|c} A(\lambda) & \hat B (\lambda)^T \\ \hline B(\lambda) & 0 \end{array}\right] \end{equation} is called a strong block minimal bases pencil if it has the following additional properties:

\begin{itemize}

\item[(a)] The row degrees of $B(\lambda)$ and $\hat{B}(\lambda)$ are all equal to one.
\item[(b)] The row degrees of any minimal basis dual to $B(\lambda)$ are all equal.
\item[(c)] The row degrees of any minimal basis dual to $\hat{B}(\lambda)$ are all equal.

\end{itemize}
\end{definition}

We will adopt the convention that if the block $B(\lambda)$ ($\hat{B}(\lambda)$) is absent, then the corresponding dual minimal basis is an identity matrix of the same size as the number of columns (rows) of $A(\lambda).$ The following theorem about block minimal bases pencils which is a combination of ~\cite[Theorems 3.3 and 3.7]{DopLPV16} will be important for the results in this section and the next one.

\begin{theorem}\label{blockmin} Let $L(\lambda)$ be a minimal bases pencil given by~\eqref{minb} and $C(\lambda)$ and $\hat{C}(\lambda)$ be the dual minimal bases of $B(\lambda)$ and $\hat{B}(\lambda)$ respectively. Then $L(\lambda)$ is a linearization of the matrix polynomial
\begin{equation}\label{blockmin_eq} Q(\lambda) = {\hat C}(\lambda)A(\lambda)C(\lambda)^T. \end{equation}
Moreover, if $L(\lambda)$ is a strong block minimal bases pencil, then the following hold.

\begin{itemize}
\item[(a)] $L(\lambda)$ is a strong linearization of $Q(\lambda)$ considered as a polynomial of grade $1 + \mathrm{deg} \, C + \mathrm{deg} \, \hat{C}.$
\item[(b)] If $0 \leq \epsilon_1 \leq \epsilon_2 \leq \cdots \leq \epsilon_p$ are the right minimal indices of $Q(\lambda),$ then $$\epsilon_1 + \deg \, C \leq \epsilon_2 + \deg \, C \leq \cdots \leq \epsilon_p + \deg \, C,$$
    are the right minimal bases of $L(\lambda).$
\item[(c)] If $0 \leq \eta_1 \leq \eta_2 \leq \cdots \leq \eta_q$ are the left minimal indices of $Q(\lambda),$ then $$\eta_1 + \deg \, {\hat C} \leq \eta_2 + \deg \, {\hat C} \leq \cdots \leq \eta_q + \deg \, {\hat C}$$ are the left minimal indices of $L(\lambda).$
\end{itemize}
\end{theorem}

Given a strong block minimal bases pencil, the above result shows the construction of a polynomial from the pencil such that the pencil is a strong linearization of the polynomial and lays out the recovery rules for extracting left and right minimal indices of the polynomial from those of the pencil. However in practice, we are generally more interested in the reverse process, i.e., given a matrix polynomial $Q(\lambda)$ of grade $k,$ we are interested in constructing a strong linearization from which the left and right minimal indices of the polynomial can be easily extracted. It was shown in~\cite{DopLPV16}, that this easily achieved by the so called Block Kronecker pencils that are a special class of strong  block minimal bases pencils for which $$\hat B(\lambda)=H_\epsilon(\lambda) \otimes I_m  \mbox{ and }B(\lambda) =H_\eta(\lambda) \otimes I_n$$ with $\epsilon + \eta + 1 = k,$ and $H_j(\lambda)$ given by~\eqref{dualminb2}.  The conditions on the block $A(\lambda)$ under which the Block Kronecker pencils become strong linearizations of a given polynomial $P(\lambda)$ are given in~\cite[Theorem 5.4]{DopLPV16}.  Now we have the main result of this section.

\begin{theorem}\label{recovlt1} Let $P(\lambda) = \sum_{i=0}^k \lambda^iA_i$ be an $m \times n$ matrix polynomial of grade $k$ with $m\geq n$ and $r = \mathrm{nrank} P(\lambda).$ Let $L(\lambda) \in \L_1(P)$ with right ansatz vector $v \in \F^k \setminus \{0\}$ be of full Z-rank. Let $L_t(\lambda)$ be the pencil obtained by trimming $L(\lambda) \in \L_1(P),$ with respect to $M$ and $D$.
Then $L_t(\lambda)$ is a strong linearization of $P(\lambda)$ such that the following hold.

\begin{itemize}

\item[(a)] Every minimal basis of $N_r(L_t)$ is of the form  $\{\Lambda_k(\lambda) \otimes x_1(\lambda), \ldots, \Lambda_k(\lambda) \otimes x_{n-r}(\lambda)\}$ where \\ $\{ x_1(\lambda), \ldots, x_{n-r}(\lambda)\}$ is a minimal basis of $N_r(P).$

\item[(b)] The right minimal indices of $L_t(\lambda)$ are those of $P(\lambda)$ shifted by $k-1.$

\item[(c)] Every minimal basis of $N_l(P)$ is of the form $\{(v^{T}\otimes I_m)D^Ty_1(\lambda), \ldots, (v^{T}\otimes I_m)D^Ty_{m-r}(\lambda)\}$ where $\{y_1(\lambda), \ldots, y_{m-r}(\lambda)\}$ is a minimal basis of $N_l(L_t)$.

\item[(d)] The left minimal indices of $P(\lambda)$ are equal to those of $L_t(\lambda).$

\end{itemize}
\end{theorem}

\begin{proof}
From~\eqref{hatlt} and \eqref{lt_hatlt},
\begin{equation}\label{stronglin} L_t(\lambda) := {\tilde D} \left[\begin{array}{c|c} I_m & \\ \hline  & \tilde{R} \end{array}\right] \underbrace{\lambda\left[\begin{array}{c|c}\alpha A_k & X_{12}\\ \hline & -I_{(k-1)n} \end{array}\right]+\left[\begin{array}{c|c} Y_{11}& \alpha A_0\\ \hline I_{(k-1)n} & \end{array}\right]}_{=: K(\lambda)}
\end{equation}

Clearly $K(\lambda)$ is in the Block Kronecker form $\left[\begin{array}{c} A(\lambda) \\ \hline B(\lambda) \end{array}\right]$ with $$A(\lambda) := \lambda\left[\begin{array}{cc} \alpha A_k & X_{12} \end{array}\right] + \left[\begin{array}{cc} Y_{11} & \alpha A_0 \end{array}\right] \mbox{ and } B(\lambda) = H_{k-1}(\lambda) \otimes I_n.$$ Now $A(\lambda)C(\lambda)^T = P(\lambda),$ where $C(\lambda) = \frac{1}{\alpha}(\Lambda_k(\lambda) \otimes I_n)^T$ is a dual of $B(\lambda).$ As the block ${\hat B}(\lambda)$ is absent in $K(\lambda)$ (and consequently, ${\hat C}(\lambda) = I_m,$) by Theorem~\ref{blockmin}, $K(\lambda)$ is   a strong linearization of $P(\lambda)$ such that $P(\lambda)$ and $K(\lambda)$ have the same left minimal indices and the right minimal indices of $K(\lambda)$ are those of $P(\lambda)$ shifted by $k-1.$ The relation~\eqref{stronglin}, shows that the same is true of each pencil $L_t(\lambda)$ obtained by trimming a strong g-linearization in $\mathbb{L}_1(P)$ and this proves (b) and (d).

The process of obtaining $L_t(\lambda)$ from $L(\lambda)$  implies that $N_r(L) = N_r(L_t).$ Therefore the proof of (a) follows from Theorem~\ref{rminbasis}.
To prove (c) we consider the map
 \begin{eqnarray*}
 \mathcal{L}:N_l(L_t) &\rightarrow& N_l(P)\\
 y(\lambda) &\mapsto& \mathcal{L}_v(D^Ty(\lambda))
 \end{eqnarray*}
 If $y(\lambda)\in N_l(L_t)$ then $D^Ty(\lambda)\in N_l(L) \subset N_l(P)$, and therefore $\mathcal{L}$ is well defined. Also from the definition of $\mathcal{L}$ it is clear that it is a linear map from $N_l(L_t)$ to $N_l(P).$ We first show that $\mathcal{L}$ is bijective. Let $Z$ be the Z-matrix of $(M \otimes I_n)L(\lambda)$ and $Q_2$ be the last $c = (k-1)(m-n)$ columns of the unitary matrix $Q$ of a $QR$ decomposition of $Z.$ Now $y(\lambda) \in N(\mathcal{L})$ if and only if $D^Ty(\lambda) \in N(\mathcal{L}_v).$
 As noted in Remark~\ref{conslminbasis}, there exists a basis of $N(\mathcal{L}_v)$ of the form
 $\left\{(M^T\otimes I_m)\begin{bmatrix}0\\\ q_1 \end{bmatrix},\dots, (M^T\otimes I_m)\begin{bmatrix}0\\q_c\end{bmatrix} \right\}$ where $c=(k-1)(m-n)$ and $\{\bar q_1,\dots, \bar q_c\}$ are the columns of $Q_2.$ Therefore there exists a nonzero $a \in \F^c$ such that $$D^Ty(\lambda) = (M^T \otimes I_m) \begin{bmatrix}I_m&\\& \overline{Q_2}\end{bmatrix}\begin{bmatrix}0\\a\end{bmatrix} = (M^T \otimes I_m) \begin{bmatrix} 0 \\ \overline{Q_2}a \end{bmatrix}.$$
This implies that
$\begin{bmatrix} y(\lambda) \\ -a \end{bmatrix}^T\begin{bmatrix}D \\ \left[\begin{array}{cc} 0 & Q_2^* \end{array}\right] (M \otimes I_m) \end{bmatrix}= 0,$ which gives $y(\lambda)=0$ as $\begin{bmatrix} D \\ \left[\begin{array}{cc} 0 & Q_2^* \end{array}\right] (M \otimes I_m)\end{bmatrix}$ is nonsingular. Therefore $\mathcal{L}$ is a one to one linear map. Since $N_l(L_t)$ and $N_l(P)$ are of the same dimension, it follows that $\mathcal{L}$ is a bijective linear map.

Now we will show for any vector polynomial $p(\lambda)\in N_l(P)$ of degree $\delta$ we can find a polynomial vector $z(\lambda)\in\ N_l(L_t)$ of degree $\delta$ such that  $\mathcal{L}(z(\lambda))=p(\lambda)$. By Lemma~\ref{mainl} $\exists$ $\hat z(\lambda)\in N_l(L)$ such that $\mathcal{L}_v(\hat z(\lambda))=p(\lambda)$ and $\mathrm{deg} \, \hat z=\mathrm{deg} \, p$. Let $z(\lambda)$ be the first with $m + (k-1)n$ entries of a vector polynomial ${\hat y}(\lambda)$ which satisfies $\begin{bmatrix}D \\ \left[\begin{array}{cc} 0 & Q_2^* \end{array}\right](M \otimes I_m)\end{bmatrix}^T{\hat y}(\lambda) = \hat z(\lambda).$ Then $z(\lambda) \in N_l(L_t)$ as,
$$z(\lambda)^TL_t(\lambda) = {\hat y}(\lambda)^T\left[\begin{array}{c} L_t(\lambda) \\ 0_c \end{array}\right] = {\hat y}(\lambda)^T\begin{bmatrix}D \\ \left[\begin{array}{cc} 0 & Q_2^*\end{array}\right](M \otimes I_m)\end{bmatrix}L(\lambda) = \hat z(\lambda)^TL(\lambda) = 0.$$

Now as $(v^T \otimes I_m)(M^T \otimes I_m)\left[\begin{array}{c} 0_{m+(k-1)n} \\ \bar{Q_2} \end{array} \right] = (\alpha e_1^T \otimes I_m)\left[\begin{array}{c} 0_{m+(k-1)n} \\ \bar{Q_2} \end{array} \right] = 0,$ we have $$\mathcal{L} (z(\lambda)) = (v^T \otimes I_m)D^Tz(\lambda) = (v^T \otimes I_m)\begin{bmatrix}D \\ \left[\begin{array}{cc} 0 & Q_2^*\end{array}\right](M \otimes I_m)\end{bmatrix}^T{\hat y}(\lambda) = \mathcal{L}_v (\hat z (\lambda)) = p(\lambda).$$
Also it is clear that $z(\lambda)$ and $p(\lambda)$  have the same degree as $$\deg \, z  \geq \deg \, p = \deg \, \hat z = \deg \, \hat y \geq \deg \, z.$$  Now the proof of part (c) follows by arguing as in the proof of Theorem~\ref{lminbasis}. \end{proof}

\begin{remark} It can be proved that the pencils $L_t(\lambda)$ are strong linearizations of $P(\lambda)$ without establishing their connection with Block Kronecker pencils. However, we prefer to give this connection to highlight their position in the current literature of strong linearizations of rectangular matrix polynomials. In fact it is also clear from~\eqref{stronglin} that among the linearizations formed by trimming g-linearizations in $\L_1(P),$ only the pencils $\hat L_t(\lambda)$ of the form~\eqref{hatlt} belong to the Block Kronecker ansatz spaces $\mathbb{G}_1(P)$ introduced in~\cite{FasS18}.
\end{remark}

\begin{remark}
The recovery rules for the left and right minimal bases of $P(\lambda)$ from those of $L_t(\lambda)$ may also be derived from~\cite[Theorem 7.7]{DopLPV16} which gives the rules for extracting the same for general Block Kronecker linearizations. However, as the pencils $K(\lambda)$ in~\eqref{stronglin} to which the pencils $L_t(\lambda)$ are strictly equivalent are special types of Block Kronecker pencils, we prefer to prove these parts directly by using the notions and techniques previously introduced in the paper.
\end{remark}

In a similar way, if $m\leq n,$ pencils in $\L_2(P)$ of full Z-rank can provide strong linearizations of $P(\lambda).$ In particular if $L(\lambda) \in \L_2(P)$ with nonzero left ansatz vector $w \in \F^k$ has full Z-rank,  then for any nonsingular matrix $\hat{M} \in \F^{k \times k}$ such that $\hat Mw = \alpha e_1$ for some $\alpha \neq 0,$ $$L(\lambda)(\hat M^T \otimes I_n) = \lambda \left[\begin{array}{c|c} \alpha A_k & \\\hline \hat{X}_{12}  & -\hat Z \end{array}\right] + \left[\begin{array}{c|c} \hat{Y}_{11} & \hat Z \\\hline \alpha A_0 &  \end{array}\right],$$
where $\hat Z \in \F^{(k-1)m \times (k-1)n}$ with $\rank \, \hat Z = (k-1)m.$ If $\hat Z^* = Q \left[\begin{array}{c} \hat R \\ 0 \end{array} \right],$ be a QR decomposition of $\hat Z^*$ and $Q_2$ is  the matrix formed by the last $c = (k-1)(n-m)$ columns of $Q,$ then it is easy to see that $$L(\lambda)(\hat M^T \otimes I_n)\left[\begin{array}{c} 0 \\ Q_2 \end{array}\right] = 0.$$ For any choice of $\hat{D} \in \F^{kn \times (n + (k-1)m)},$ such that the $kn \times kn$ matrix
\begin{equation}\label{choiceofhatd}
\left[\begin{array}{cc} \hat D & (\hat M^T \otimes I_n) \left[\begin{array}{cc} 0 & Q_2^T \end{array}\right]^T \end{array}\right]
\end{equation}
is nonsingular, we get the pencils $L(\lambda) \hat D.$ We refer to them as \emph{the pencils formed by trimming $L(\lambda) \in \L_2(P),$ with respect to $\hat M$ and $\hat D.$} For instance, the second companion linearization $C_2(\lambda)$ arises from $C_2^g(\lambda) \in \L_2(P),$ with respect to $\hat M = I_k$ and $\hat D = \left[\begin{array}{cc} I_n & \\ & I_{k-1} \otimes I_{n,m} \end{array}\right].$

By arguing as in the proof of Theorem~\ref{recovlt1}, these pencils can be shown to be strictly equivalent to Block Kronecker linearizations of $P(\lambda)$ of the form $\left[\begin{array}{cc} A(\lambda) & H_{k-1}(\lambda)^T \otimes I_m \end{array}\right]$ from which the left and right minimal bases and indices of $P(\lambda)$ may be easily extracted. In fact we have the following theorem.

\begin{theorem}\label{recovlt2} Let $P(\lambda) = \sum_{i=0}^k \lambda^iA_i$ be an $m \times n$ matrix polynomial of grade $k$ with $m\leq n$ and $r = \mathrm{nrank} P(\lambda).$ Let $L(\lambda) \in \L_2(P)$ with left ansatz vector $w \in \F^k \setminus \{0\}$ be of full Z-rank. Let $L_t(\lambda)$ be the pencil formed by trimming $L(\lambda) \in \L_2(P),$ with respect to $\hat M$ and $\hat D$. Then $L_t(\lambda)$ is a strong linearization of $P(\lambda)$ and the following hold.

\begin{itemize}

\item[(a)] Every minimal basis of $N_l(L_t)$ is of the form  $\{\Lambda_k(\lambda) \otimes y_1(\lambda), \ldots, \Lambda_k(\lambda) \otimes y_{m-r}(\lambda)\}$ where \\ $\{ y_1(\lambda), \ldots, y_{m-r}(\lambda)\}$ is a minimal basis of $N_l(P).$

\item[(b)] The left minimal indices of $L_t(\lambda)$ are those of $P(\lambda)$ shifted by $k-1.$

\item[(c)] Every minimal basis of $N_r(P)$ is of the form $\{(w^{T}\otimes I_n)\hat Dx_1(\lambda), \ldots, (w^{T}\otimes I_n)\hat Dx_{n-r}(\lambda)\}$ where $\{x_1(\lambda), \ldots, x_{n-r}(\lambda)\}$ is a minimal basis of $N_r(L_t)$.

\item[(d)] The right minimal indices of $P(\lambda)$ are equal to those of $L_t(\lambda).$

\end{itemize}
\end{theorem}

\begin{remark} For $L(\lambda) \in \mathbb{L}_2(P)$ of full Z-rank with left ansatz vector $w \in \F^k \setminus \{0\}$ and a given choice of ${\hat M} \in \F^{k \times k}$ such that ${\hat M}w = \alpha e_1,$ if the matrix $\hat D$ in~\eqref{choiceofhatd} is chosen to be the first $n + (k-1)m$ columns of $({\hat M}^T \otimes I_n) \left[\begin{array}{cc} I_n & \\ & Q \end{array}\right],$ then in fact, the resulting pencil belongs to the Block Kronecker ansatz space $\mathbb{G}_k(P)$ introduced in~\cite{FasS18}. Also every other pencil formed by trimming $L(\lambda)$ with respect to $\hat M$ and some other choice of $\hat D$ is strictly equivalent to such a pencil but does not belong to $\mathbb{G}_k(P).$
\end{remark}

As the following example shows, the linearizations $L_t(\lambda)$ arising from the pencils of full Z-rank in $\L_1(P)$ and $\L_2(P)$ are not subclasses of the class of block minimal bases linearizations.

 \begin{example} Consider $P(\lambda) = \lambda^2 A_2 + \lambda A_1 + A_0$ where $$A_2 = \left[\begin{array}{cc} 1 & 2 \\ 2 & 5 \\ 4 & 9 \end{array}\right], \, A_1 = \left[\begin{array}{cc} 3 & 4 \\ 9 & 2 \\ 15 & 10 \end{array}\right], \mbox{ and } A_0 = \left[\begin{array}{cc} 1 & 7 \\ 2 & 5 \\ 4 & 19 \end{array}\right].$$
 Then $L(\lambda) = \lambda \hat X + \hat Y \in \L_1(P)$ where $$\hat X = \left[\begin{array}{cccc} 0 & 0 & -1 & 0 \\ 0 & 0 & 0 & -1 \\ 0 & 0 & 0 & 0 \\ 1 & 2 & 0 & 0 \\ 2 & 5 & 0 & 0 \\ 4 & 9 & 0 & 0 \end{array}\right], \, \hat Y = \left[\begin{array}{cccc} 1 & 0 & 0 & 0 \\ 0 & 1 & 0 & 0 \\ 0 & 0 & 0 & 0 \\ 3 & 4 & 1 & 7 \\ 9 & 2 & 2 & 5 \\ 15 & 10 & 4 & 19 \end{array} \right],$$ with corresponding right ansatz vector vector $v = \left[\begin{array}{c} 0 \\ 1 \end{array}\right].$ Now $\underbrace{\left[\begin{array}{cc} 0 & 1 \\ 1 & 0 \end{array}\right]}_{= : M}v = \left[\begin{array}{c} 1 \\ 0 \end{array}\right],$ and $$(M \otimes I_3)L(\lambda) = \lambda X + Y,$$ where
 $$ X  =  (M \otimes I_3)\hat X = \left[\begin{array}{cccc} 1 & 2 & 0 & 0 \\ 2 & 5 & 0 & 0 \\ 4 & 9 & 0 & 0 \\ 0 & 0 & -1 & 0 \\ 0 & 0 & 0 & -1 \\ 0 & 0 & 0 & 0 \end{array}\right], \mbox{ and }
 Y  =  (M \otimes I_3)\hat Y = \left[\begin{array}{cccc} 3 & 4 & 1 & 7 \\ 9 & 12 & 2 & 5 \\ 15 & 10 & 4 & 19 \\ 1 & 0 & 0 & 0 \\ 0 & 1 & 0 & 0 \\ 0 & 0 & 0 & 0 \end{array}\right].$$
 Clearly, $Z = \left[\begin{array}{cc} 1 & 0 \\ 0 & 1 \\ 0 & 0 \end{array}\right]$ is full rank with QR decomposition $Z = I_3 Z.$ Hence, $$\left[\begin{array}{cc} 0 & Q_2^* \end{array}\right](M^T \otimes I_3) = \left[\begin{array}{cccccc} 0 & 0 & 1 & 0 & 0 & 0 \end{array}\right], \mbox{ and } \left[\begin{array}{cccccc} 1 & 0 & 0 & 0 & 0 & 0 \\ 0 & 1 & 0 & 0 & 0 & 0 \\ 0 & 0 & 0 & 1 & 0 & 0 \\ 0 & 0 & 0 & 0 & 1 & 0 \\ 0 & 0 & 0 & -2 & -1 & 1 \\ \hline  0 & 0 & 1 & 0 & 0 & 0 \end{array}\right]$$ is nonsingular.  So, $$L_t(\lambda) := \left[\begin{array}{cccccc} 1 & 0 & 0 & 0 & 0 & 0 \\ 0 & 1 & 0 & 0 & 0 & 0 \\ 0 & 0 & 0 & 1 & 0 & 0 \\ 0 & 0 & 0 & 0 & 1 & 0 \\ 0 & 0 & 0 & -2 & -1 & 1 \end{array}\right] L(\lambda) = \lambda \left[\begin{array}{cccc} 0 & 0 & -1 & 0 \\ 0 & 0 & 0 & -1 \\ 1 & 2 & 0 & 0 \\ 2 & 5 & 0 & 0 \\ 0 & 0 & 0 & 0 \end{array}\right] + \left[\begin{array}{cccc} 1 & 0 & 0 & 0 \\ 0 & 1 & 0 & 0 \\ 3 & 4 & 1 & 7 \\ 9 & 2 & 2 & 5 \\ 0 & 0 & 0 & 0 \end{array}\right],$$ is a strong linearization of $P(\lambda).$ Evidently it is not a block minimal bases linearization.
 \end{example}

\begin{remark}
 It is clear that if $m\geq n$ the size  $(m+(k-1)n)\times kn$ of $L_t(\lambda)$ is the same as that of the first Frobenius companion linearization $C_1(\lambda)$ of $P(\lambda).$  On the other hand, if $m\leq n,$ then $L_t(\lambda)$ is of size $km \times (n + (k-1)m)$ which is the same as that of the second Frobenius companion linearization $C_2(\lambda).$ Since $C_1(\lambda)$ and $C_2(\lambda)$ are the smallest among all possible Fiedler and Block Kronecker linearizations of $P(\lambda),$ therefore, the size of $L_t(\lambda)$ is less than or equal to that of all such linearizations for rectangular matrix polynomials.
 \end{remark}

\begin{remark}
Although the pencils $L_t(\lambda)$ are extracted from g-linearizations in $\L_1(P)$ and $\L_2(P),$ in practice, it is not necessary to form them by trimming g-linearizations. For example we can directly build these linearizations from a given matrix polynomial $P(\lambda) = \sum_{i=0}^k \lambda^i A_i,$ of grade $k$ and $m > n,$  by using the fact that they are of the form $L_t(\lambda) := \tilde D \left(\lambda\left[\begin{array}{c|c}\alpha A_k & X_{12}\\ \hline & -\tilde{R}\end{array}\right]+\left[\begin{array}{c|c} Y_{11}& \alpha A_0\\ \hline \tilde{R} & \end{array}\right] \right)$ where $\alpha \neq 0,$ $X_{12}+Y_{11}=\alpha\begin{bmatrix}A_{k-1}&\dots&A_1\end{bmatrix}$ and ${\tilde D} \in \F^{m + (k-1)n \times m + (k-1)n}$ and $\tilde R \in \F^{(k-1)n \times (k-1)n}$ are nonsingular. The process of trimming g-linearizations to form linearizations of this type can be seen as a means to connect the g-linearizations of $P(\lambda)$ with linearizations.
\end{remark}

 In the next section we undertake a global backward stability analysis of the solution of polynomial eigenvalue problems using $L_t(\lambda)$ on the lines of the analysis in~\cite{DopLPV16} and show that their is in fact a wide choice of optimal strong linearizations (beyond the ones identified in~\cite{DopLPV16}) which can be used to solve the complete eigenvalue problem for $P(\lambda)$ in a globally backward stable manner.

\section{Global backward error analysis of solutions of polynomial eigenvalue problems using linearizations arising from g-linearizations}

In this section we carry out a global backward error analysis of the process of solving the complete eigenvalue problem associated with a rectangular matrix polynomial $P(\lambda) = \sum_{i = 0}^k \lambda^i A_i \in \F[\lambda]^{m \times n}$ of grade $k$ by using linearizations that arise from a g-linearization in $\mathbb{L}_1(P)$ or $\mathbb{L}_2(P).$ It will be an extension of the one in~\cite{DopLPV16} for Block Kronecker linearizations. As mentioned in Section~\ref{intro}, any solution of such a problem involves finding the finite and infinite eigenvalues and associated elementary divisors as well as the left and right minimal bases and indices of $P(\lambda).$ Typically this is done by initially finding the said quantities for some choice of strong linearization via very effective backward stable methods like the {\it staircase algorithm} proposed in~\cite{Van79} and further developed in~\cite{DemK93a, DemK93b}. The backward stability of such algorithms guarantee that any computed solution of the eigenvalue problem corresponding to a linearization say, $L(\lambda)$ of $P(\lambda),$ is the exact solution of the problem for a pencil $L(\lambda) + \Delta L(\lambda)$ where $\frac{\normp{\Delta L}}{\normp{L}} = O(\bf{u})$ with respect to some norm $\normp{\cdot}.$ The solution of the complete eigenvalue problem for $P(\lambda)$ is then computed from the solution for $L(\lambda) + \Delta L(\lambda)$ by applying the same recovery rules to $L(\lambda) + \Delta L(\lambda)$ that would have been applied to the solution for $L(\lambda)$ if it were available. Following~\cite{DopLPV16}, the process is said to be globally backward stable if it is the  exact solution of the complete eigenvalue problem for $P(\lambda) + \Delta P(\lambda)$ with the following conditions being met.

\begin{itemize}

 \item[(a)] If $P(\lambda)$ is of grade $k,$  the perturbed pencil $L(\lambda) + \Delta L(\lambda)$ a strong linearization of $P(\lambda) + \Delta P(\lambda)$ of grade $k$ such that $\frac{\normp{\Delta P }}{\normp{ P }} = O({\bf u}).$
 \item[(b)] The rules for extracting the left and right minimal indices of $P(\lambda)$ from those of $L(\lambda)$ remain the same when they are replaced by $P(\lambda) + \Delta P(\lambda)$ and $L(\lambda) + \Delta L(\lambda)$ respectively.
 \end{itemize}

 The analysis in~\cite{DopLPV16}, showed that (a) and (b) are satisfied for optimal choices of Block Kronecker linearization of $P(\lambda)$ with respect to the norm $$\normp{ P }_F := \sqrt{\sum_{i = 0}^k \| A_i\|_F^2}$$ where $\| A \|_F := \sqrt{\mathrm{trace} \, (A^*A)}$ is the Frobenius norm of $A.$ In particular, it was shown that there exists a constant $C_{P, L}$ depending on $P(\lambda)$ and $L(\lambda)$ such that \begin{equation}\label{ubd16} \frac{\normp{ \Delta P}_F}{\normp{P}_F} \leq C_{P, L} \frac{\normp{ \Delta L }_F}{\normp{ L }_F},\end{equation} where, $C_{P, L} \approx k^3\sqrt{m + n}$ under certain conditions that are satisfied by appropriate choice of Block Kronecker linearizations and scaling of $P(\lambda).$

 We establish that the same analysis can be extended to solutions obtained via linearizations $L_t(\lambda)$ of $P(\lambda) \in \F[\lambda]^{m \times n}$ that arise from g-linearizations in $\L_1(P)$ when $m > n.$ Similar arguments can easily complete the corresponding analysis for the case $m < n$ with respect to linearizations that arise from g-linearizations in $\mathbb{L}_2(P).$

\noindent Our choice of norm  $\normp{ P }_F$ on $\F[\lambda]^{m \times n}$ considered as a vector space over $\F$  is not submultiplicative.  The following  lemma from \cite{DopLPV16} which bounds the Frobenius norm of the product of two matrix polynomials will therefore be useful in the analysis. For notational convenience in this section we set  $$\Lambda_{k,p}(\lambda):=(\Lambda_k(\lambda)\otimes I_p), PQ(\lambda):=P(\lambda)Q(\lambda) \mbox{ and } (P+Q)(\lambda):=P(\lambda)+Q(\lambda),$$ for two matrix polynomials $P(\lambda)$ and $Q(\lambda)$ for which the above products and sums are defined.

\begin{lemma}\label{normp_F}
 Let $P(\lambda)=\sum_{i=0}^{d_1}A_i\lambda^i$ and $Q(\lambda)=\sum_{i=0}^{d_2}B_i\lambda^i$ be two matrix polynomials and such that all the products below are defined. Then the following inequalities hold.

 \begin{enumerate}
  \item $\normp{PQ}_F \leq \min{\{\sqrt{d_1+1},\sqrt{d_2+1}}\}\normp{P}_F\normp{Q}_F$
  \item $\normp{P\Lambda_{k,p}}_F \leq \min{\{\sqrt{d_1+1},\sqrt{k}}\}\normp{P}_F$
 \end{enumerate}

\end{lemma}

Initially we analyse the global backward stability of the process of computing a solution of the complete eigenvalue problem for $P(\lambda)$ arising from linearizations of the form~\eqref{hatlt}. Later on we will extend this analysis to the case where any linearization $L_t(\lambda)$ arising from a g-linearization in $\mathbb{L}_1(P)$
is used.


 Since the matrix $\tilde R \in \F^{(k-1)n}$ of $\hat L_t(\lambda)$ given by~\eqref{hatlt} is upper triangular and nonsingular, $\hat L_t(\lambda)$ is a strong block minimal bases pencil of the form $$\left[\begin{array}{c} A(\lambda) \\ \hline B(\lambda) \end{array}\right]$$ where,
\begin{eqnarray}
A(\lambda) & = & \lambda \left[\begin{array}{cc} \alpha A_k & X_{12}
\end{array}\right],\label{M1} + \left[\begin{array}{cc} Y_{11} & \alpha A_0 \end{array}\right],\label{A}\\
B(\lambda) & = & \lambda \left[\begin{array}{cc} 0_{(k-1)n \times n} & -\tilde{R} \end{array}\right] +
\left[\begin{array}{cc} \tilde{R} & 0_{(k-1)n \times n} \end{array}\right].\label{B}
\end{eqnarray}
Note that $\Lambda_{k,n}(\lambda)^T$ is a dual minimal basis of $B(\lambda).$

Any computed solution of the complete eigenvalue problem associated with $\hat{L}_t(\lambda)$ is an exact solution of a perturbed pencil $\hat L_t(\lambda) + \Delta \hat L_t(\lambda)$ where $$\Delta \hat L_t(\lambda)=\left[\begin{array}{c}\Delta A(\lambda)\\ \hline \Delta B(\lambda)\end{array}\right]$$ with  $\Delta A(\lambda)\in\mathbb{F}[\lambda]^{m\times kn}\text{ and }\Delta B(\lambda)\in\mathbb{F}[\lambda]^{(k-1)n\times kn}$ so that $$\hat{L}_t(\lambda) + \Delta \hat L_t(\lambda) = \left[\begin{array}{c} A(\lambda) + \Delta A(\lambda) \\ \hline B(\lambda) + \Delta B(\lambda) \end{array}\right].$$ Our initial aim is to show that for small enough $\normp{\Delta \hat{L}_t}_F,$ $\hat{L}_t(\lambda) + \Delta \hat{L}_t(\lambda)$ is a strong block minimal bases linearization of some perturbed polynomial $P(\lambda) + \Delta P(\lambda)$ of grade $k$ such that $\frac{\normp{\Delta P}_F}{\normp{P}_F}$ is bounded above by a small multiple of $\frac{\normp{\Delta \hat{L}_t}_F}{\normp{\hat{L}_t}_F}.$

 We establish an upper bound on $\normp{\Delta B}_F$ such that ${\hat L}_t(\lambda) + \Delta {\hat L}_t(\lambda)$ is a strong block minimal bases pencil. This requires that the following conditions are satisfied.

\noindent
{\bf Condition (A)} $B(\lambda) + \Delta B(\lambda)$ is a minimal basis with all row degrees equal to one;

\noindent
{\bf Condition (B)} There exists a matrix polynomial $\Delta D(\lambda) \in \F[\lambda]^{kn \times n}$ of grade $k-1$ such that $\Lambda_{k,n}(\lambda)^T + \Delta D(\lambda)^T$ is a dual minimal basis of $B(\lambda) + \Delta B(\lambda)$ with all row degrees equal to $k-1.$

Following the strategy in~\cite{DopLPV16}, we will use the concept of convolution matrices associated with $P(\lambda)= \sum_{i=0}^{k}A_i\lambda^i$  which are defined as follows.
\begin{equation}\label{conv_matrix}
C_j(P)=\underbrace{\begin{bmatrix}
  A_k&&&\\
  A_{k-1} & A_k &&\\
  \vdots  & A_{k-1} & \ddots &\\
  A_0 & \vdots & \ddots & A_k\\
  & A_0 & & A_{k-1}\\
  &&\ddots & \vdots\\
  &&& A_0
 \end{bmatrix}}_{j+1\text{ block columns}} \mbox{ for } j=0,1,2,\dots.
\end{equation}

The following lemma which states some important and useful properties of convolution matrices can be easily proved.

\begin{lemma}\label{convmprop} Let $P(\lambda)= \sum_{i=0}^{k}A_i\lambda^i$ and $Q(\lambda)= \sum_{i=0}^{l}B_i\lambda^i$ be matrix polynomials of grade $k$ and $l$ respectively and $C_j(P)$ and $C_j(Q)$ for $j=0,1,2,\dots,$ be corresponding convolution matrices as given by~\eqref{conv_matrix}.
\begin{itemize}
 \item[(a)] If $P(\lambda)$ and $Q(\lambda)$ are of same size and grade then $C_j(P + Q)=C_j(P)+C_j(Q)$ for all $j$.
 \item[(b)] $\norm{C_j(P)}_F=\sqrt{j+1}\normp{P}_F$ for all $j$.
 \item[(c)] If the product $P(\lambda)Q(\lambda)$ is defined, then considering it as a grade $k+l$ matrix polynomial, we have $C_0(PQ)=C_{l}(P)C_0(Q)$.
\end{itemize}
\end{lemma}

The next Theorem from~\cite{DopLPV16} for convolution matrices will be useful to show that for sufficiently small $\normp{\Delta B }_F,$ $B(\lambda) + \Delta B(\lambda)$ can be a minimal basis with all row degrees equal to $1.$

\begin{theorem}\label{convmain} For any positive integer $l,$ let $N(\lambda) = A + \lambda B \in \F[\lambda]^{ln \times (l+1)n}$ and $C_j(N)$ for $j = 0, 1, \ldots,$ be the sequence of convolution matrices of $N(\lambda).$ Then $N(\lambda)$ is a minimal basis with all its row degrees equal to $1$ and all the row degrees of any dual minimal basis equal to $l,$ if and only if $C_{l-1}(N) \in \F^{(l+1)l n \times (l+1) l n}$ is nonsingular and $C_l(N) \in \F^{l(l+2)n \times (l+1)^2 n}$ has full row rank.
\end{theorem}

Observing that $B(\lambda) = \tilde R (H_{k-1}(\lambda) \otimes I_n)$ where $H_{k-1}(\lambda)$ is given by~\eqref{dualminb2},
the next lemma which is proved in the appendix will be useful in establishing a bound on $\normp{\Delta B}_F$ that achieves the desired objectives.~\footnote{A proof of this result is available in~\cite{DopLPV18} which is a revised version of~\cite{DopLPV16}. Our proof was made independently and with different arguments.}

\begin{lemma}\label{tautheorem}
For $\tau(\lambda) := (H_{k-1}(\lambda) \otimes I_n),$  $\sigma_{\min}(C_{k-2}(\tau))= \sigma_{\min}(C_{k-1}(\tau))=2\sin(\frac{\pi}{4k-2})\geq\frac{3}{2k}$.
\end{lemma}

The following result bounds $\normp{\Delta B}_F$ such that $B(\lambda) + \Delta B(\lambda)$ is a minimal basis with all row degrees equal to $1.$

\begin{theorem}\label{L2_minimalbasis}
   Let $B(\lambda)$ be the pencil given by~\eqref{B}, and $\Delta B(\lambda)\in \mathbb{F}[\lambda]^{(k-1)n\times kn}$ be any pencil such that
   \begin{equation}\label{normdeltadin} \normp{\Delta B}_F < \displaystyle{\frac{3 \sigma_{\min}(\tilde{R})}{2k^{3/2}}}. \end{equation}  Then $B(\lambda)+\Delta B(\lambda)$ is a minimal basis with all its row degrees equal to $1$ and all row degrees of any minimal basis dual to it equal to $k-1$.
\end{theorem}

\begin{proof}
In view of Theorem~\ref{convmain},  the proof follows by establishing that $C_{k-2}(B+\Delta B)$ is nonsingular and $C_{k-1}(B+\Delta B)$ has full row rank.
 For $ j = k-1 \mbox{ or } k-2,$
\begin{equation} \sigma_{\min}(C_j(B)) =  \sigma_{\min}\left(\begin{bmatrix}\tilde R&&\\&\ddots&\\&&\tilde R\end{bmatrix}C_j(\tau)\right)
 \geq \sigma_{\min}(\tilde R)\sigma_{\min}(C_j(\tau)).
 \end{equation}
  Therefore by Lemma~\ref{tautheorem},
  \begin{equation}\label{sigmincjk}
  \sigma_{\min}(C_j(B)) \geq 2\sigma_{\min}(\tilde R)\sin{\frac{\pi}{4k-2}} \geq \sigma_{\min}(\tilde{R})\frac{3}{2k}.
  \end{equation}
  Since $\tilde R$ is nonsingular, it follows that $C_{k-2}(B)$ is nonsingular and $C_{k-1}(B)$ has full row rank. By Lemma~\ref{convmprop}(a),
  $$C_j(B+\Delta B)=C_j(B)+C_j(\Delta B)$$ for $j = k-1$ and $k-2.$ Therefore  $C_{k-2}(B+\Delta B)$ is nonsingular and $C_{k-1}(B+\Delta B)$ has full row rank if $\norm{C_j(\Delta B)}_F<\sigma_{\min}(C_j(B))$  for both values of $j.$ But both inequalities follow from Lemma~\ref{convmprop}(b), and the relations~\eqref{normdeltadin} and~\eqref{sigmincjk}. Hence the proof. \end{proof}
%

Now the following result establishes the required upper bound on $\normp{\Delta B}_F$ such that both {\bf Condition (A)} and {\bf Condition (B)} are fulfilled. The proof is omitted as it follows by arguing as in the proof of~\cite[Theorem 6.18]{DopLPV16}.

\begin{theorem}\label{existanceof_delD}
 Let $B(\lambda)$ be the pencil given by~\eqref{B} and $\Delta B(\lambda)\in \mathbb{F}[\lambda]^{(k-1)n\times kn}$ be any pencil such that
 \begin{equation}\label{normdeltad} \normp{\Delta B}_F<\frac{ \sigma_{\min}(\tilde R)}{2k^{3/2}}. \end{equation}
  Then there exists a matrix polynomial $\Delta D(\lambda)\in \mathbb{F}[\lambda]^{kn\times n}$ of grade $k-1$ such that
 \begin{itemize}
  \item[(a)] $B(\lambda)+\Delta B(\lambda)$ and $\Lambda_{k,n}(\lambda)^T+\Delta D(\lambda)^T$ are dual minimal bases, with all the row degrees equal to $1$ and $k-1$ respectively, and
  \item[(b)] $\normp{\Delta D}_F\leq \frac{k\sqrt{2}}{\sigma_{\min}(\tilde R)}\normp{\Delta B}_F<\frac{1}{\sqrt{2k}}$.
 \end{itemize}

\end{theorem}

Next we have the main result which completes the global backward error analysis for solutions of the complete eigenvalue problem for $P(\lambda)$ obtained from the linearizations $\hat L_t(\lambda).$

\begin{theorem}\label{main_LThat}
Let ${\hat L}_t(\lambda)$ be any linearization of $P(\lambda) = \sum_{i=0}^k \lambda^iA_i \in \F[\lambda]^{m \times n}$ of grade $k$ with $m > n,$ of the form~\eqref{hatlt}.
Let $A(\lambda)$ and $B(\lambda)$ be the blocks of ${\hat L}_t(\lambda)$ as specified by~\eqref{A} and \eqref{B} respectively and $\tilde{R} \in \F^{m + (k-1)n}$ be the nonsingular upper triangular matrix appearing in the block $B(\lambda).$
  If $\Delta \hat L_t(\lambda)$ is any pencil of the same size as $\hat L_t(\lambda)$ such that
  \begin{equation}\label{finalbndthat} \normp{\Delta \hat L_t}_F<\frac{\sigma_{\min}(\tilde R)}{2k^{3/2}} \end{equation}
  then $\hat L_t(\lambda)+\Delta \hat L_t(\lambda)$ is a strong linearization of a matrix polynomial $P(\lambda)+\Delta P(\lambda)$ of grade k and
  \begin{equation}\label{finaleq}  \frac{\normp{\Delta P}_F}{\normp{P}_F}\leq \hat C_{\hat L_t, P} \frac{\normp{\Delta \hat L_t}_F}{\normp{\hat L_t}_F} \end{equation}
  where $\hat C_{\hat L_t,P} =\frac{1}{\lvert\alpha\rvert}\frac{\normp{\hat L_t}_F}{\normp{P}_F}\left(3+2k\frac{\normp{A}_F}{\sigma_{\min}(\tilde R)}\right)$.

  The right minimal indices of $\hat L_t(\lambda)+\Delta \hat L_t(\lambda)$ are those of $P(\lambda)+\Delta P(\lambda)$ shifted by $k-1$ and left minimal indices of $\hat L_t(\lambda)+\Delta \hat L_t(\lambda)$ are same as those of $P(\lambda)+\Delta P(\lambda)$, which is the same as the corresponding relationship between the minimal indices of $\hat L_t(\lambda)$ and $P(\lambda)$.
 \end{theorem}

 \begin{proof} Clearly, \begin{equation*}
  \normp{\Delta \hat L_t}_F<\frac{\sigma_{\min}(\tilde R)}{2k^{3/2}}
  \Rightarrow \normp{\Delta B}_F<\frac{\sigma_{\min}(\tilde R)}{2k^{3/2}}.
   \end{equation*}
   By Theorem~\ref{existanceof_delD}, there exists $\Delta D(\lambda) \in \mathbb{F}[\lambda]^{kn\times n}$ of grade $k-1$ such that $B(\lambda)+\Delta B(\lambda)$ and $\Lambda_{k,n}(\lambda)^T+\Delta D(\lambda)^T$ are dual minimal bases with all the row degrees $1$ and $k-1$ respectively. Therefore $\hat{L}_t(\lambda) + \Delta \hat{L}_t(\lambda)$ is a strong block minimal bases pencil and Theorem~\ref{blockmin} implies that $\hat{L}_t(\lambda) + \Delta \hat{L}_t(\lambda)$ is a strong block minimal bases linearization of $$\frac{1}{\alpha}
(A(\lambda) + \Delta A(\lambda))(\Lambda_{k.n}(\lambda) + \Delta D(\lambda)) := P(\lambda) + \Delta P(\lambda)$$ of grade $k.$ As $P(\lambda)=\frac{1}{\alpha}A(\lambda)(\Lambda_{k,n}(\lambda)),$
   we have, $$\Delta P(\lambda)=\frac{1}{\alpha}\{(A(\lambda)+\Delta A(\lambda))\Delta D(\lambda)+\Delta A(\lambda)(\Lambda_{k,n}(\lambda))\}$$
   which implies that
  $$\normp{\Delta P}_F\leq \frac{1}{\lvert\alpha\rvert}\{\normp{A(\Delta D)}_F+\normp{(\Delta A)(\Delta D)}_F+\normp{(\Delta A)\Lambda_{k,n}}_F\}.$$
  By applying Lemma~\ref{normp_F} and Theorem~\ref{existanceof_delD}~(b), we get

  \begin{equation} \frac{\normp{\Delta P}_F}{\normp{P}_F}\leq \frac{1}{\lvert\alpha\rvert}\frac{\normp{\hat L_t}_F}{\normp{P}_F}\left(3+2k\frac{\normp{A}_F}{\sigma_{\min}(\tilde R)}\right)\frac{\normp{\Delta \hat L_t}_F}{\normp{\hat L_t}_F}
    \end{equation}

Also as the block $\hat B(\lambda)$ is absent in the linearization $\hat{L}_t(\lambda) + \Delta \hat{L}_t(\lambda),$ we have $\hat C(\lambda) = I_m$ in~\eqref{blockmin_eq} and consequently by Theorem~\ref{blockmin}, the right minimal indices of $\hat L_t(\lambda)+\Delta \hat L_t(\lambda)$ are those of $P(\lambda)+\Delta P(\lambda)$ shifted by $k-1$ and left minimal indices of $\hat L_t(\lambda)+\Delta \hat L_t(\lambda)$ are same as those of $P(\lambda)+\Delta P(\lambda).$ By Theorem~\ref{recovlt1}, the shifting relations between the left and right minimal indices of $P(\lambda) + \Delta P(\lambda)$ and $\hat{L}_t(\lambda) + \Delta \hat{L}_t(\lambda)$ are exactly the same as those between $P(\lambda)$ and $\hat{L}_t(\lambda).$
\end{proof}


Now we extend the above analysis to solutions of the complete eigenvalue problem for $P(\lambda)$ obtained via any linearization $L_t(\lambda)$ arising from a g-linearization in $\mathbb{L}_1(P).$ As noted in Section~\ref{lin}, any such linearization $L_t(\lambda)$ is strictly equivalent to a linearization of the form $\hat L_t(\lambda).$ Using this fact, and the results for $\hat L_t(\lambda),$ we have the following theorem.

\begin{theorem}\label{main_LT}

Let $L_t(\lambda)$ be any linearization of $P(\lambda) = \sum_{i=0}^k \lambda^iA_i \in \F[\lambda]^{m \times n}$ of grade $k$ with $m > n,$ arising from a g-linearization in $\mathbb{L}_1(P).$
  Let ${\hat L}_t(\lambda) = {\tilde D}^{-1}L_t(\lambda)$ where $\tilde D$ is as given in~\eqref{lt_hatlt}.
  Then ${\hat L}_t(\lambda)$ is of the form~\eqref{hatlt}. Let $A(\lambda)$ and $B(\lambda)$ be the blocks of ${\hat L}_t(\lambda)$ as specified by~\eqref{A} and \eqref{B} respectively and $\tilde{R} \in \F^{m + (k-1)n}$ be the nonsingular upper triangular matrix appearing in the block $B(\lambda).$ If $\Delta L_t(\lambda)$ be any pencil of the same size as $L_t(\lambda)$ such that
    \begin{equation}\label{finalbndt} \normp{\Delta L_t}_F<\frac{\sigma_{\min}(\tilde R)\sigma_{\min}(\tilde D)}{2k^{3/2}}, \end{equation}
  then $L_t(\lambda)+\Delta L_t(\lambda)$ is a strong linearization of a matrix polynomial $P(\lambda)+\Delta P(\lambda)$ as grade k such that
  \begin{equation}\label{finaleq2} \frac{\normp{\Delta P}_F}{\normp{P}_F}\leq C_{L_t,P} \frac{\normp{\Delta L_t}_F}{\normp{ L_t}_F}, \end{equation}
  where $C_{L_t,P}=\frac{\kappa_2(\tilde D)}{\lvert\alpha\rvert}\frac{\normp{\hat L_t}_F}{\normp{P}_F}\left(3+2k\frac{\normp{A}_F}{\sigma_{\min}(\tilde R)}\right),$ $\kappa_2({\tilde D})$ being the 2-norm condition number of $\tilde D.$ The right minimal indices of $ L_t(\lambda)+\Delta  L_t(\lambda)$ are those of $P(\lambda)+\Delta P(\lambda)$ shifted by $k-1$ and left minimal indices of $ L_t(\lambda)+\Delta  L_t(\lambda)$ are same as those of $P(\lambda)+\Delta P(\lambda)$, which is the same as the corresponding relations between the minimal indices of $L_t(\lambda)$ and $P(\lambda)$.
 \end{theorem}

 \begin{proof} Evidently, ${\hat L}_t(\lambda) := {\tilde D}^{-1}L_t(\lambda)$ is of the form form~\eqref{hatlt}. Since $L_t(\lambda) + \Delta L_t(\lambda) = {\tilde D}({\hat L}_t(\lambda) + {\tilde D}^{-1} \Delta L_t (\lambda)),$ and~\eqref{finalbndt} implies that
 $$\normp{{\tilde D}^{-1}\Delta L_t}_F<\frac{\sigma_{\min}(\tilde R)}{2k^{3/2}},$$ by Theorem~\ref{main_LThat}, ${\hat L}_t(\lambda) + {\tilde D}^{-1} \Delta L_t (\lambda)$ is a strong block minimal bases linearization of some polynomial $P(\lambda) + \Delta P(\lambda)$ of grade $k$ such that
 \begin{equation} \frac{\normp{\Delta P}_F}{\normp{P}_F}\leq \frac{1}{\lvert\alpha\rvert}\frac{\normp{\hat L_t}_F}{\normp{P}_F}\left(3+2k\frac{\normp{A}_F}{\sigma_{\min}(\tilde R)}\right)\frac{\normp{{\tilde D}^{-1} \Delta L_t}_F}{\normp{\hat L_t}_F}. \end{equation}

The relation~\eqref{finaleq2} now follows by using the fact that $L_t(\lambda)=\tilde D\hat L_t(\lambda).$ Also as $L_t(\lambda) + \Delta L_t(\lambda) = {\tilde D}({\hat L}_t(\lambda) + {\tilde D}^{-1}\Delta L_t(\lambda)),$ ${\hat L}_t(\lambda) + {\tilde D}^{-1} \Delta L_t (\lambda)$ are strictly equivalent, $L_t(\lambda) + \Delta L_t(\lambda)$ is a strong linearization of $P(\lambda) + \Delta P(\lambda)$ and the recovery rules for the left and right minimal indices of $P(\lambda) + \Delta P(\lambda)$ from those of $L_t(\lambda) + \Delta L_t(\lambda)$ are the same as the ones from ${\hat L}_t(\lambda) + {\tilde D}^{-1} \Delta L_t (\lambda).$ Therefore it follows from Theorem~\ref{main_LThat}, that the right minimal indices of $ L_t(\lambda)+\Delta  L_t(\lambda)$ are those of $P(\lambda)+\Delta P(\lambda)$ shifted by $k-1$ and left minimal indices of $ L_t(\lambda)+\Delta  L_t(\lambda)$ are same as those of $P(\lambda)+\Delta P(\lambda)$, which is same as the corresponding relations between the minimal indices of $L_t(\lambda)$ and $P(\lambda)$.
\end{proof}
%
%

Observe that as the block $\hat B(\lambda)$ is absent in the pencil $\hat{L}_t(\lambda)$ when compared with the block minimal bases pencil~\eqref{minb}, this greatly simplifies the analysis as the arguments in pages 24-30 of~\cite{DopLPV16} for Block Kronecker linearizations  may be skipped as a consequence.

If the complete eigenvalue problem for $L_t(\lambda)$ is solved by using a backward stable algorithm, then $\frac{\normp{\Delta L_t}_F}{\normp{L_t}_F} = O({\bf u}).$  In such a situation~\eqref{finaleq2} shows that the process of solving the complete eigenvalue problem for $P(\lambda)$ via linearizations $L_t(\lambda),$ is globally backward stable if $C_{ L_t,P}$ is not very large. As $C_{L_t, P} =\kappa_2(\tilde D) {\hat C}_{\hat L_t, P}$, so a good choice of $L_t(\lambda)$ would be one for which $\kappa_2(\tilde D) \approxeq 1$ and ${\hat C}_{\hat L_t, P}$ is not large for the corresponding pencil $\hat L_t(\lambda) = {\tilde D}^{-1}L_t(\lambda).$ To identify such linearizations, we first note that for the block $A(\lambda)$ of ${\hat L}_t(\lambda),$
$$\alpha P(\lambda)=A(\lambda)\Lambda_{k,n}(\lambda)\Rightarrow \left|\alpha\right|\normp{ P}_F=\normp{A\Lambda_{k,n}}_F\leq \sqrt{2}\normp{A}_F.$$ This implies that
 $$\frac{\normp{\hat L_t}_F}{\normp{P}_F} \geq \frac{\normp{A}_F}{\normp{P}_F} \geq \frac{\left|\alpha\right|}{\sqrt{2}}.$$

Now if $\left|\alpha\right|\normp{P}_F>>\sigma_{\min}(\tilde R)$ then $\frac{\sqrt{2}\normp{A}_F}{\sigma_{\min}(\tilde R)}>>1$, and since $\frac{\normp{\hat L_t}_F}{\left|\alpha\right|\normp{P}_F}\geq \frac{1}{\sqrt{2}}$ so $\hat C_{\hat L_t,P}$ will be big. Again if $\left|\alpha\right|\normp{P}_F<<\sigma_{\min}(\tilde R)$ then
$\frac{{\normp{\hat L_t}}_F}{\left|\alpha\right|\normp{P}_F} > \frac{{\|\tilde R\|}_F}{\left|\alpha\right|\normp{P}_F}>>1$ and once again $\hat C_{\hat L_t,P}$ will be big. So, a good choice of $\hat L_t(\lambda)$ would be one for which $\left|\alpha\right|\normp{P}_F\approxeq\sigma_{\min}(\tilde R)$.

Besides, if $\normp{A}_F \approxeq \left|\alpha\right|\normp{P}_F\approxeq \sigma_{\min}(\tilde R)$, and $\kappa_2(\tilde R) \approxeq 1$ then $$C_{L_t, P} \approxeq (3+2k)\sqrt{1+2(k-1)n}$$ and then $$\frac{\normp{\Delta P}_F}{\normp{P}_F}\lessapprox (3+2k)\sqrt{1+2(k-1)n}\frac{\normp{\Delta L_t}_F}{\normp{ L_t}_F}.$$

In summary, by using linearizations $L_t(\lambda)$ satisfying
\begin{itemize}
\item[(i)] $\kappa_2(\tilde{D}) \approxeq 1$ and $\kappa_2(\tilde R) \approxeq 1$ and
\item[(ii)] $\normp{A}_F \approxeq \left|\alpha\right|\normp{P}_F\approxeq \sigma_{\min}(\tilde R)$,
\end{itemize}
we will have $\frac{\normp{\Delta P}_F}{\normp{P}_F} = O({\bf u})$ if $\frac{\normp{\Delta L_t}_F}{\normp{L_t}_F} = O({\bf u}).$ So the complete eigenvalue problem for $P(\lambda)$ can be solved in a globally backward stable manner by using backward stable algorithms to solve the complete eigenvalue problem for such choices of $L_t(\lambda).$
The optimal Block Kronecker linearizations of the form $\left[\begin{array}{c} A(\lambda) \\ B(\lambda) \end{array}\right]$ ensuring global backward stability that were identified in~\cite{DopLPV16} are included in the above choices. In fact they are the ones for which $\tilde D = I_{m + (k-1)n},$ $|\alpha| = 1/\normp{P}_F,$ $\tilde{R} = I_{(k-1)n}$ and $\|X_{12}\|_F^2 + \|Y_{11}\|_F^2 \approxeq \frac{1}{{\normp{P}_F}^2}\sum_{i=1}^{k-1}\|A_i\|_F^2$ in~\eqref{stronglin} which include the Frobenius companion form $C_1(\lambda).$ Our analysis shows that there exist many more choices of linearizations from among the pencils $L_t(\lambda)$ with which the complete eigenvalue problem for $P(\lambda)$ can be solved in a globally backward stable manner.

\section{Conclusion} Given an $m \times n$ rectangular matrix polynomial $P(\lambda) = \sum_{i=0}^k \lambda^iA_i,$ of grade $k,$ in this paper we have introduced the notion of a generalized linearization (g-linearization) of $P(\lambda).$  We have also constructed vector spaces of rectangular matrix pencils such that almost every matrix pencil in the space provides solutions of the complete eigenvalue problem for $P(\lambda)$ with the property that the left and right minimal indices and bases of $P(\lambda)$ can be easily extracted from those of the pencil. These spaces become the vector spaces $\L_1(P)$ and $\L_2(P)$ introduced in~\cite{MacMMM06a} whenever $P(\lambda)$ is square. They also have the same properties with respect to g-linearizations that the spaces $\L_1(P)$ and $\L_2(P)$ have with respect to linearizations (as shown in~\cite{DeDM09}) when $P(\lambda)$ is square and singular. The results provide a direct extension of the theory of the vector spaces $\L_1(P)$ and $\L_2(P)$ to the case of rectangular matrix polynomials. We have also shown a process of extracting many different strong linearizations from almost every pencil in $\L_1(P)$ and $\L_2(P).$ We believe that our work complements the recent work in~\cite{FasS18} which allows the study of linearizations of rectangular matrix pencils in a vector space setting  by introducing the Block Kronecker ansatz spaces. While~\cite{FasS18} gives the relationship between the particular Block Kronecker ansatz spaces $\mathbb{G}_1(P)$ and $\mathbb{G}_k(P)$ and the spaces $\L_1(P)$ and $\L_2(P)$ respectively when $P(\lambda)$ is square and regular, our work extends the notion of the spaces $\L_1(P)$ and $\L_2(P)$ to the rectangular case and shows the relationship between strong linearizations of $P(\lambda)$ extracted from the pencils in $\L_1(P)$ and $\L_2(P)$ and the linearizations in $\mathbb{G}_1(P)$ and $\mathbb{G}_k(P).$

A global backward error analysis of the process of solving the complete eigenvalue problem for $P(\lambda)$ via the linearizations that can be extracted from strong g-linearizations in $\L_1(P)$ and $\L_2(P)$ was also conducted on the lines of the one in~\cite{DopLPV16}. It showed that these g-linearizations provide a wide choice of linearizations that can solve the eigenvalue problem for $P(\lambda)$ in a globally backward stable manner. This analysis which was not carried out earlier even for the case that $P(\lambda)$ is square, will be useful in making optimal choices of linearizations in computation. Moreover, we believe that when $P(\lambda)$ is square and has some additional structure, the results may be extended to identify larger collections of optimal structure preserving linearizations beyond the ones known in the literature, with respect to which the eigenvalue problem for $P(\lambda)$ can be solved in a globally backward stable manner.

\section{Appendix: Proof of Lemma~\ref{tautheorem}}
 As the result is obvious for $k=2$, we assume that $k> 2$. Since $C_j(\tau) = C_j(H_{k-1}) \otimes I_n$ for $j = k-1, k-2,$ it is enough to show that \begin{equation}\label{sigmin} \sigma_{\min}(C_{k-1}(H_{k-1})) = \sigma_{\min}(C_{k-2}(H_{k-1})) = 2\sin\left(\frac{\pi}{4k-2}\right).\end{equation}
 For simplicity, we denote $C_j(H_{k-1})$ by $C_j$ for $j = k-1, k-2.$ To complete the proof it is enough to show that both the matrices $C_{k-1}$ and $C_{k-2}$ are full rank and the smallest nonzero eigenvalues of $S_{k-2} := C_{k-2}^*C_{k-2}$ and $S_{k-1} := C_{k-1}^*C_{k-1}$ are both equal to $2 + 2 \cos \left(\frac{2(k-1)\pi}{2k-1}\right).$ For $j \geq 1,$ let
 \begin{center}
 $D_j=\begin{bmatrix}1&&&&\\&2&&&\\&&\ddots&&\\&&&2&\\&&&&1\end{bmatrix}_{j \times j},
 L_j=\begin{bmatrix}0&&&\\-1&0&&\\&\ddots&\ddots&\\&&-1&0\end{bmatrix}_{j \times j} \mbox{ and,}$\\ $T_j = D_j + L_j + L_j^T + e_je_j^T, {\hat T}_j = D_j + L_j + L_j^T + e_1e_1^T,$
 \end{center}
 where $e_j$ is the $j$-th column of $I_j.$
 Now a simple multiplication shows that $$S_{k-2}=\underbrace{\begin{bmatrix}D_k&L_k^T&&&&\\L_k&D_k&L_k^T&&&\\&\ddots&\ddots&\ddots&&\\&&\ddots&\ddots&\ddots&\\&&&L_k&D_k&L_k^T\\&&&&L_k&D_k\end{bmatrix}}_{(k-1)\text{ block columns.}} \mbox{ and } S_{k-1}=\underbrace{\begin{bmatrix}D_k&L_k^T&&&&\\L_k&D_k&L_k^T&&&\\&\ddots&\ddots&\ddots&&\\&&\ddots&\ddots&\ddots&\\&&&L_k&D_k&L_k^T\\&&&&L_k&D_k\end{bmatrix}}_{k\text{ block columns.}}.$$
%
%
%
 To find the smallest eigenvalue of $S_{k-2}$ consider the permutation matrix $$\mathcal{\tilde P}=\begin{bmatrix} \mathcal{\tilde P}_1& \mathcal{\tilde P}_2&\dots& \mathcal{\tilde P}_{k+1}\end{bmatrix}\in \mathbb{C}^{k(k-1)\times k(k-1)}$$ where $ \mathcal{\tilde P}_1=\begin{bmatrix}{\tilde e}_1&{\tilde e}_{1+(k+1)}&\dots&{\tilde e}_{1+(k-2)(k+1)}\end{bmatrix}$,
 $\mathcal{\tilde P}_2=\begin{bmatrix}{\tilde e}_2&{\tilde e}_{2+(k+1)}&\dots&{\tilde e}_{2+(k-2)(k+1)}\end{bmatrix}$ and \\
 $\mathcal{\tilde P}_i=\begin{bmatrix}{\tilde e}_{i}&{\tilde e}_{i+(k+1)}&\dots&{\tilde e}_{i+(k-3)(k+1)}\end{bmatrix}$, for $i=3,4,\dots,(k+1)$. Here ${\tilde e}_j$ is the $j$-th column of $I_{k(k-1)}$. Then $ \mathcal{\tilde P}^TS_{k-2}\mathcal{\tilde P}$ is a block diagonal matrix of $(k+1)$ blocks where the first block is $T_{k-1},$ the second block is ${\hat T}_{k-1}$, and
the $j$-th block is $\begin{bmatrix}{\hat T}_{k-j+1} & \\ & T_{j-3}\end{bmatrix}$ for $j = 3,4,\dots,(k+1)$ with $T_0$ and $\hat{T}_0$ being empty matrices.

 Clearly the first two blocks have the same eigenvalues and the sub-blocks of all other blocks are submatrices of the first or second block. Hence the smallest eigenvalue of $S_{k-2}$ is the smallest eigenvalue of any one of the first $2$ blocks, in particular of the second block ${\hat T}_{k-1}.$

 To find the smallest eigenvalue of $S_{k-1},$ consider the permutation matrix $$\mathcal{\hat P}=\begin{bmatrix}\mathcal{\hat P}_1&\mathcal{\hat P}_2&\dots& \mathcal{\hat P}_{k+1}\end{bmatrix}\in \mathbb{C}^{k^2\times k^2}$$
 where $ \mathcal{\hat P}_1=\begin{bmatrix}{\hat e}_1&{\hat e}_{1+(k+1)}&\dots&{\hat e}_{1+(k-1)(k+1)}\end{bmatrix}$ and $ \mathcal{\hat P}_i=\begin{bmatrix}{\hat e}_{i}&{\hat e}_{i+(k+1)}&\dots&{\hat e}_{i+(k-2)(k+1)}\end{bmatrix}$, for $i=2,3,\dots,(k+1)$. Here ${\hat e}_j$ is the $j$-th column of $I_{k^2}$. Then $\mathcal{\hat P}^TS_{k-1}\mathcal{\hat P}$ is a block diagonal matrix of $(k+1)$ blocks where  the
first block is $D_k + L_k + L_k^T,$ the second block is ${\hat T}_{k-1},$ and
 the $j$-th block is
 $\begin{bmatrix}{\hat T}_{k-j+1} & \\ & T_{j-2}\end{bmatrix}$ for $j=3,4,\dots,(k+1)$ with ${\hat T}_0$ being the empty matrix.

 Clearly $0$ is an eigenvalue of the first block and the second block can be obtained by removing the first row and first column of the first block. Hence the smallest eigenvalue of the second block is less than or equal to the second smallest eigenvalue of the first block. Again the sub-blocks of all other blocks are either submatrices of the second block ${\hat T}_{k-1}$  or of $T_{k-1}.$ Since $T_{k-1}$ and ${\hat T}_{k-1}$ have the same eigenvalues, the smallest eigenvalue of $S_{k-1}$ is $0$ and the second smallest eigenvalue is the smallest eigenvalue of the second block ${\hat T}_{k-1}.$

 From \cite[Theorem~1]{Yue05} the smallest eigenvalue of ${\hat T}_{k-1}$ is $2+2\cos\left(\frac{2(k-1)\pi}{2k-1}\right) (\neq 0)$.
 Hence $2+2\cos\left(\frac{2(k-1)\pi}{2k-1}\right)$ is the smallest nonzero eigenvalue of both the matrices $S_{k-1}$ and $S_{k-2}.$ This proves that $C_{k-1}$ and $C_{k-2}$ are both full rank such that~\eqref{sigmin} holds. \hfill{$\square$}

\bibliographystyle{plain}
\bibliography{lin_recv2}

\end{document}